\documentclass[11pt]{article}
\PassOptionsToPackage{dvipsnames}{xcolor}
\usepackage[utf8]{inputenc}
\usepackage{tikz}
\usetikzlibrary{shapes,arrows,positioning}
\usetikzlibrary{automata,arrows,positioning,calc}

\usepackage{amsmath}
\usepackage{amssymb}
\usepackage{amsthm}
\usepackage{hyperref}
\usepackage{bbm}
\usepackage{blkarray}
\usepackage{algorithmic}
\usepackage{caption}
\usepackage{subcaption}
\usepackage{float}
\usepackage{stmaryrd}
\usepackage{enumitem}

\usepackage{fancyhdr}
\usepackage{amsopn}

\usepackage{bbm}

\usepackage{booktabs}

\hypersetup{colorlinks,
citecolor=orange,
citebordercolor=magenta,
linkcolor=magenta
}

\providecommand{\keywords}[1]
{
  \small	
  \textbf{{Keywords.}} #1
}

\providecommand{\AMS}[1]
{
  \small	
  \textbf{{AMS subject classifications.}} #1
}

\usepackage[linesnumbered,ruled,vlined]{algorithm2e}

\numberwithin{equation}{section}

\newtheorem{theorem}{Theorem}[section]
\newtheorem{remark}[theorem]{Remark}
\newtheorem{definition}[theorem]{Definition}
\newtheorem{lemma}[theorem]{Lemma}
\newtheorem{proposition}[theorem]{Proposition}

\DeclareMathOperator*{\argmin}{arg\,min}

\title{Mean Field Game Model for an \\Advertising Competition in a Duopoly}

\author{Ren\'e Carmona\footnote{Department of Operations Research and Financial Engineering,
  Princeton University, 
  Princeton, NJ 08544 
  (\href{mailtorcarmona@princeton.edu}{rcarmona@princeton.edu},
  \href{mailto:gokced@princeton.edu}{gokced@princeton.edu})}
\and G\"ok\c ce Dayan{\i}kl{\i}
    \footnotemark[1]
}
\date{}

\allowdisplaybreaks
\begin{document}

\maketitle
\begin{abstract}
In this study, we analyze an advertising competition in a duopoly. We consider two different notions of equilibrium.
We model the companies in the duopoly as major players, and the consumers as minor players. In our first game model we identify Nash Equilibria (NE) between all the players. Next we frame the model to lead to the search for Multi-Leader-Follower Nash Equilibria (MLF-NE). This approach is reminiscent of Stackelberg games in the sense that the major players design their advertisement policies assuming that the minor players are rational and settle in a Nash Equilibrium among themselves.
This rationality assumption reduces the competition between the major players to a 2-player game. After solving these two models for the notions of equilibrium, we analyze the similarities and differences of the two different sets of equilibria.
\end{abstract}

\vskip3mm

\vskip3mm
\AMS{ 91A80, 91A16, 91A07
}

\vskip 12pt\noindent
\emph{\textbf{Funding.} Both authors were supported by AFOSR award FA9550-19-1-0291.
}
{}

\section{Problem Statement and Literature Review}

	In this paper, we analyze an advertising competition in a duopoly with special attention paid to consumer behavior. We consider a model with $2$ large companies which we regard as \emph{major players}, and a large number of consumers whom we treat as \emph{minor players}. Company $j$, $j = 1,2$ produces product $j$, and product differentiation is horizontal which means that even if the quality of the products are the same, they are differentiated in the consumers' perception. Therefore, at the same price, some consumers prefer Product 1, while others prefer the other product. The model is designed as a static (one-shot) game.
	
	As expected in a duopoly, one of the goals of the companies is to increase their sales. According to \cite{Bass2005}, a company in a duopoly can increase its sales either by increasing its market share or if there is a market expansion. We assume that there is no market expansion; in other words, the total sales of the products stay constant. This is reasonable since we work with a static model instead of a dynamic one. Since we are assuming that the market size is constant, one of the goals of the companies is to increase their market share. \cite{Doyle1968} states that market share can be increased either by decreasing the price, or increasing the advertising. He also mentions that in a market with few companies, competition is through non-price ways.

	\cite{Mankiw2012} states that companies in oligopoly with differentiated consumer products such as soft drink, perfume or breakfast cereal have incentive to invest in advertising to make the consumers less price elastic. One of the goals of advertisement is to convince consumers that companies' products are more differentiated than they really are. Therefore,  advertisements are more often than not persuasive instead of informative, trying to create a brand name and foster brand loyalty. On the other hand, consumers may perceive advertisement as a signal of quality, and this may make them likely to prefer the highly advertised product. Therefore, persuasive advertisements affect consumer's preference by boosting the product's perceived value.

	According to \cite{Clarke1973}, advertisement of a company does not only affect their bottom line, it also affects the opposing company. When a company advertises more, it increases its own sales and decreases the opposing firm's sales. Therefore, a company would like to increase its relative advertising which is the ratio of their own advertisement efficiency to the total advertisement efficiency. This is an instance of \emph{negative externality} as increasing one of the product's advertisement efficiency leads to a decrease in the opposing company's demand.

	 It is particularly hard to find Nash equilibria in games with large numbers of players. However, by assuming a form of symmetry among the players' behaviors, and letting the number of players go to infinity while the influence of each individual player fades, we can make use of the recently developed theory of Mean Field Games (MFGs). Mean Field Game models were introduced by \cite{Lasry2007}, and independently by \cite{Huang2003,Huang2004,Huang2006}.

The history of the subject and the development of the probabilistic approach to the solution of Mean Field Games introduced in \cite{Carmona2013} and further information can be found in the two volume book of \cite{carmona_delarue_2018}. While MFGs are relevant in plenty of practical situations, in many real life applications there exists a player that affects the system disproportionally, for example a government or a regulator. In these cases, the addition of a \emph{major player} may be required. Mean Field Games with major and minor players were introduced by \cite{Nourian2013} and analyzed by \cite{Carmona2016}, \cite{Carmona2017}, and \cite{Bensoussan2016}. In these types of games, the minor players' decisions are affected by the aggregate of the other minor players as well as the decision of the major player. On the other hand, the decision of the latter is only affected by it own costs and rewards, and aggregate statistics from the population of the minor players. The originality of our contribution is to consider the competition between two major players affecting the field of minor players in a way akin to what was considered in the literature we just cited.

	In this paper, we analyze two different equilibrium notions for advertisement and product consumption levels in a duopoly. In the first case, a Nash equilibrium between both major players and the consumers is analyzed. In the second one, the major players compete in a 2-player game assuming that the consumers are rational, and anticipating their purported behaviors. They choose their advertising policies assuming that the consumers will react to their choices and settle in a Nash equilibrium among themselves. We call this equilibrium ``Multi-Leader-Follower Nash Equilibrium". So for this second equilibrium notion, the major players compete among each others, but vis-a-vis the consumers, they behave as in a Stackelberg game by taking actions assuming that the minor players will react rationally.
	Even though leaders and  followers do not act contemporaneously in the original Stackelberg game model introduced by Heinrich Stackelberg in 1934, this will be the case in the first of our models. Note also that a model of a Multi-Leader-Follower game for $N$ followers was analyzed in \cite{Fukushima2015}, but the mean field limit $N\to\infty$ and the subsequent Mean Field Game formulation for the minor players were not considered. 
	
	We call the first model setting where we search for a Nash equilibrium among all major and minor players \textit{``MFG Formulation for a Nash Equilibrium (NE)"}, and the second setting where we search for an equilibrium when the minor players are settling in a Nash equilibrium among themselves while reacting to the major players who are playing a 2-player game  \textit{``MFG Formulation for a Multi-Leader-Follower Nash Equilibrium (MLF-NE)"}. 
	
	With this model, we conclude that for companies, it is inefficient to use Nash Equilibrium advertising strategies instead of using Multi-Leader-Follower Nash Equilibrium strategies. The reason for this is that companies overly advertise and consequently incur high costs, if they are not able to understand how the consumers are going to react to their strategies (NE setup). Therefore, it is recommended that they should understand the consumers' behavior and use the MLF-NE strategies. Further, we also deduce that a company in an adverse position initially (i.e. having a lower market share at the beginning) may end up as a market leader, if the companies are able to analyze consumers' reaction and in other words, use MLF-NE strategies. However, if the companies are using NE strategies while advertising, the market leader protects its position.
	
	 Advertising behavior of one major player with a large number of minor players is analyzed in \cite{Malhame2016_2}. However, in that paper, the model is dynamic and there is no competition among major players. Competition in terms of price and quantity in an Oligopoly by using Mean Field Games was analyzed by Chan and Sircar \cite{Chan2015}. In this model, consumers are not included as players and a large number of firms are set as players; moreover, the competition between them is not in terms of advertising. Therefore, our model is the first model that analyzes advertising competition in a duopoly under the Mean Field Games paradigm with multiple major and minor players. 	

The paper is structured as follows. First, we introduce the model with $N$ consumers and articulate the equilibrium notions in Section~\ref{sec:N_player_model}. Then we give the mean field game formulation in Section~\ref{sec:mfg_model}. We state amd prove our existence and uniqueness results for both equilibrium notions in Section~\ref{sec:theoretical_results}. Finally we compare the properties of these two equilibrium notions through numerical experiments in Section~\ref{sec:Numerical_results}.

\section{N-Player Model}
\label{sec:N_player_model}
	\subsection{Minor Players}
		We first consider the case of a finite number of consumers and we assume they behave in a \emph{symmetric} manner. A generic consumer (minor player) is denoted as minor $i$ where $i=1,2,...,N$.
		
		 Each consumer $i$ controls their preference rate for Product 1 which is denoted as $u_i^{c} \in [0,1]$. In particular, if $u^{c}_{i} = 1$, then consumer $i$ buys Product 1 only, whereas, if $u^{c}_{i} = 0$, they buy Product 2 only. Whenever $u_i^{c} \in (0,1)$, their consumption of Product 1 is $(100\times u_i^{c})\%$ of their total consumption. Like for the type of a player in a Bayesian game, we assume that the initial value $u^c_{0,i}$ of the control of consumer $i$ is random and has a-priori distribution $\mu_0\in\mathcal{P}([0,1])$ where $\mathcal{P}([0,1])$ denotes the set of probability measures on $[0,1]$. Each player knows their initial preference, but does not know the others. We shall assume that the actual control $u^c_i$ will be a feedback function of its initial value.
		 
		 Given the major players' advertisement efficiencies and the empirical distribution of the other minor players' controls, each consumer decides on their own control according to their goals and costs. Because of our symmetry assumption, we assume that all the consumers have the same objectives. Firstly, they want to be faithful to their initial preferences, so they do not want to change their initial choices by much. However, consumers care about the choices of others and they do not want to deviate from the average, so they want to buy the more commonly preferred product. Finally, they want to increase their total utility from the products. With all these conditions in mind, we define the optimization problem of consumer $i$ as follows:
		\begin{multline} \label{MinorOpt}
		\min_{u^{c}_{i}: [0,1] \rightarrow [0,1]} \mathbb{E}_{u_{0,i}^c \sim \mu_0}
		\Bigg\{
		\dfrac{\beta}{2} (u^{c}_{i}-u^{c}_{0,i})^{2} + 
		\dfrac{\eta}{2} (u^{c}_{i}-\bar{u}_{-i}^{c})^{2}
		-\\ \left[(\alpha+e_{1})u^{c}_{i} + (\alpha+e_{2})(1-u^{c}_{i})-\dfrac{(u^{c}_{i})^{2}+(1-u^{c}_{i})^{2}}{2}\right]\Bigg\}
		\end{multline}
		 where $\bar{u}_{-i}^{c} = \frac{1}{N-1} \sum_{j=1, j\neq i}^{N}u^{c}_{j}$ and $\mu_0 \in \mathcal{P}([0,1])$. The rationale for the choice of the above objective function can be explained as follows:
		
		 The first term represents the unwillingness of a consumer to change preference. This may be caused from brand loyalty or not being prone to change. The second term represents the fact that a typical consumer does not want to deviate from the average: $\bar{u}_{-i}^{c}$ denotes the mean of the controls of other consumers, and can be interpreted as the market share of Product 1 when the number of players is large. Here $ \beta > 0 $ and $ \eta > 0$ represent the relative importance given to these cost terms. In {what} follows we use $ \beta = \eta = 1 $ for simplicity. The last part is the maximization of the utility of a consumer from the consumed products. We use the utility function already used by \cite{Hattori2012}, and previously by \cite{Singh1984} and \cite{Garella2008}:
		\begin{equation}
		\mathit{U(u_i^c)}=(\alpha+e_{1})u^{c}_{i} + (\alpha+e_{2})(1-u^{c}_{i})-\dfrac{(u^{c}_{i})^{2}+(1-u^{c}_{i})^{2}-2\gamma(1-u^{c}_{i})u^{c}_{i}}{2}
		\end{equation}
		Here, $\gamma \in [0,1]$ represents the substitutability degree of the two products: as it becomes closer to 1, consumers become more price elastic. Since we want consumers to be perfectly price inelastic, we assume $\gamma=0$. We also assume that the true qualities of Product 1 and 2 are the same, and we denote their common value by $\alpha \geq 0$. The number $e_{j}$ denotes the perceived incremental quality as a result of the advertisement of Product $j$. Here we assume:
		\begin{equation}
		\begin{aligned}
		e_1 = \mathit{f}(u_{1})\qquad\qquad
		e_2 = \mathit{f}(u_{2})
		\end{aligned}
		\end{equation}
		with $\frac{\partial e_1}{\partial u_1} \geq 0$ and $ \frac{\partial e_2}{\partial u_2} \geq 0$. This intuitively means that the utility from Product $j$ increases with the advertisement efficiency of Product $j$ and the advertisement of the opposing firm does not have an effect on consumer's perceived quality for the Product $j$ . For the sake of simplicity, we are defining $e_1$ and $e_2$ as follows:
		\begin{equation}
		\begin{aligned}
		e_1 = u_{1}\qquad\qquad
		e_2 = u_{2}
		\end{aligned}	
		\end{equation}

	\subsection{Major Players}
	
		 The two major players are the competitive companies in the duopoly. Major players $1$ and $2$ produce Product 1 and Product 2 which are not differentiated in quality but are horizontally differentiated in the perception of the consumers. This can be understood as the famous example of the Pepsi and Coke advertisement competition: even if people fail blind-folded test, they continue to like Coke over Pepsi, or the opposite. 				
		
		 According to \cite{Doyle1968}, companies in a duopoly tend to compete with non-price means. Therefore in our model, companies are not controlling the price. They may compete in terms of loyalty schemes, quality differentiation or advertisement. Since advertising is one of the main forms of competition in a duopoly, in this model both major players control their advertisement efficiency which we take as the square root of the amount they spend on advertisement and which we denote by $u_{j} \in \mathbb{R_+}$ for major {player} $j$, $j=1,2$. Here, advertisement is persuasive and it does not have predefined targets. In other words, it affects every consumer in the same way and hopefully, positively.
		
		 Major players have similar goals and costs. Firstly, they want to maximize their market share, and secondly, they want to advertise relatively more than the opposing company to be better known by consumers.  Finally, they want to minimize their cost of advertisement. We define their optimization problems as follows:
		 
		 For major {player} 1:
		\begin{equation} \label{Major1Opt}
		\min_{u_{1} \in \mathbb{R_+}} J^1(u_1; u_2, \bar u^c)=\min_{u_{1} \in \mathbb{R_+}} \Big\{ -\left(\rho_{1}u_{1}\left(1-\bar{u}^c\right)-\rho_{2}u_{2}\bar{u}^c\right)-\left(\dfrac{u_{1}+\epsilon}{u_{2}+\epsilon}\right) + \dfrac{c}{2}u_{1}^{2}\Big\}
		\end{equation}		
		
		 For major {player} 2:
		\begin{equation} \label{Major2Opt}
		\min_{u_{2} \in \mathbb{R_+}} J^2(u_2; u_1, \bar u^c)=\min_{u_{2} \in \mathbb{R_+}} \Big\{ -\left(\rho_2u_{2}\bar{u}^c-\rho_{1}u_{1}\left(1-\bar{u}^c\right)\right)-\left(\dfrac{u_{2}+\epsilon}{u_{1}+\epsilon}\right) + \dfrac{c}{2}u_{2}^{2} \Big\}
		\end{equation}
		where $\bar u^c = \frac{1}{N}\sum_{i=1}^N u_i^c$. Here is the rationale for these choices:
		
		 The first part of the cost function is for increasing their own market share. As previously stated, market share of Product 1 is taken as the mean of the control of minor players and denoted here as $\bar{u}^c=\frac{1}{N}\sum_{j=1}^{N}u_j^c$. Assuming that both companies own the entire market and that there is no market expansion, market share of Product 2 is given by $1-\bar{u}^c$. Here, we use ideas from the classical dynamic Lanchester Model used by \cite{Fruchter1997} and \cite{Fruchter1999}. Lanchester Combat Model is a competitive extension for the Vidale-Wolfe Model proposed by \cite{Vidale1957} and used by \cite{Bass2005}. Different modifications of this model are also used by \cite{Erickson1995}, and \cite{Prasad2003,Prasad2004}. According to the Lanchester model, over time the dynamics of market share are given respectively for Product 1 and 2 by:			
		\begin{equation}
		\begin{aligned}
		\dot{x} = \rho_{1}u_{1}(1-x)-\rho_{2}u_{2}(x) \\
		\dot{(1-x)} = \rho_{2}u_{2}(x) - \rho_{1}u_{1}(1-x)
		\end{aligned}
		\end{equation}
		where $x$ denotes the market share of Product 1 and $(1-x)$ denotes the market share of Product 2. Here, $\rho_{j}u_{j}$ denotes advertisement efficiency of major {player} $ j $, $j=1,2$ where $\rho_{j}$ is the positive efficiency constant and $u_{j}$ is the square root of advertisement amount of major {player} $j$. Intuitively, each company takes a part of the opposing company's market share which is proportional to their advertisement efficiency, and at the same time, each of them is losing a part of their own market share proportionally to the opposing company's advertisement efficiency.  For the sake of simplicity, in the remaining of this paper, we take $\rho_{1}=\rho_{2}=1$ and $u_1$ and $ u_2$ are called advertisement efficiencies of major player 1 and 2 respectively. The above equation gives the dynamics of the market share for Product 1 over time. However, since our model is designed as a static game, we assume that companies are focusing on increasing their market share instantly. Since they are minimizing, they are taking the negative signed versions of above change rates.
		
		 The second part of the cost function comes from the desire to be known more widely by increasing their relative advertisement efficiency. For example, major {player} 1 tries to increase the ratio $(u_{1}+\epsilon)/(u_{2}+\epsilon)$, where $\epsilon>0$ is a constant. Since the cost function is minimized we use again negative sign for this part. Here the addition of the constant $\epsilon$ is to enable the analysis of the cases where a company does not advertise, namely $u_1=0$ or $u_2=0$.

		 The last contribution to the cost is intended to minimize advertisement spending. Here $\frac{c}{2} > 0$ gives the cost per unit of advertisement and $u_{j}^{2}$ gives the advertisement amount of major {player} $ j $. For the sake of simplicity, we assume that both companies have the same unit advertisement cost. Therefore major {player} $j$ tries to minimize $\frac{c}{2}u_{j}^{2}$.
	
\subsection{Equilibrium Notions}

As explained in the introduction, we analyze two different types of equilibrim.

\begin{definition}[{Nash Equilibrium}]\label{def:nash}
With the same notations as in the previous definition, a strategy profile $\left(u^{c*}_1,u^{c*}_2,\dots,u^{c*}_N, u^{*}_1,u^*_2\right) \in S$ is called a Nash Equilibrium if:
\begin{itemize}[label=$\rightarrow$]
	\item For any fixed $ 1\leq k \leq N $, for all $ u^c_k \in S^c_k $, we have: $$J^c_k\big(\textcolor{Bittersweet}{u^c_k}; \boldsymbol{u^{c*}_{-k}},u^*_1,u^*_2\big) \geq J^c_k\big(u^{c*}_k;\boldsymbol{u^{c*}_{-k}},u^*_1,u^*_2\big), $$
	\item For all $u_1 \in S_1$, we have:
	$$J^1\left(\textcolor{Bittersweet}{u_1};\boldsymbol{u^{c*}}, u^*_2\right) \geq J^1\left(u^*_1;\boldsymbol{u^{c*}},u^*_2\right),$$
	\item For all $u_2 \in S_2$, we have:
	$$J^2\left(\textcolor{Bittersweet}{u_2};\boldsymbol{u^{c*}}, u^*_1\right) \geq J^2\left(u^*_2;\boldsymbol{u^{c*}},u^*_1\right), $$		
	where $\boldsymbol{u^{c*}}= (u^{c*}_1,\dots,u^{c*}_N)$ and $\boldsymbol{u^{c*}_{-k}}= (u^{c*}_1,\dots, u^{c*}_{k-1}, u^{c*}_{k+1}, \dots,u^{c*}_N)$.
\end{itemize}

\end{definition}

\begin{definition}[{Multi-Leader-Follower Nash Equilibrium}]\label{def:mlf_nash}
Assume there exist N many minor players and 2 major players. Let $ S = S^c_1 \times S^c_2 \times \dots S^c_N \times S_1 \times S_2$ and $J_k^c\big(u^c_k(u_1,u_2);\boldsymbol{u^c_{-k}}(u_1,u_2), u_1, u_2\big)$,\\$ J^1\left(u_1; u^c_1(u_1,u_2),...,u^c_N(u_1,u_2),u_2\right),$ $ J^2\left(u_2; u^c_1(u_1,u_2),...,u^c_N(u_1,u_2), u_1\right)$, $\forall k = 1,2,\dots N$ are strategy profiles and cost functions for N minor players and 2 major players, respectively. Then a strategy profile $\left(u^{c*}_1(u^*_1,u^*_2),u^{c*}_2(u^*_1,u^*_2),\dots,u^{c*}_N(u^*_1,u^*_2), u^{*}_1,u^*_2\right) \in S$ is called a Multi-Leader-Follower Nash Equilibrium if:
\begin{itemize}[label=$\rightarrow$]
	\item For any fixed $ 1\leq k \leq N $, for all $ u^c_k \in S^c_k $, we have: 
	$$ J^c_k\big(\textcolor{Bittersweet}{u^c_k}(u^*_1,u^*_2); \boldsymbol{u^{c*}_{-k}}(u^*_1,u^*_2),u^*_1,u^*_2\big) \geq J^c_k\big(u^{c*}_k(u^*_1,u^*_2),\boldsymbol{u^{c*}_{-k}}(u^*_1,u^*_2),u^*_1,u^*_2\big), $$
	\item For all $u_1 \in S_1$, we have:
	$$J_1\left(\textcolor{Bittersweet}{u_1};\boldsymbol{u^{c*}}(\textcolor{Bittersweet}{u_1},u^*_2),u^*_2\right) \geq J_1\left(u^*_1;\boldsymbol{u^{c*}}(u^*_1,u^*_2),u^*_2\right), $$
	\item For all $u_2 \in S_2$, we have:
	$$J_2\left(\textcolor{Bittersweet}{u_2} ; \boldsymbol{u^{c*}}(u^*_1,\textcolor{Bittersweet}{u_2}),u^*_1\right) \geq J_2\left(u^*_2,\boldsymbol{u^{c*}}(u^*_1,u^*_2),u^*_1\right),$$		
	where $\boldsymbol{u^{c*}}\left(\textcolor{Bittersweet}{u_1},u^*_2) = (u^c_1(\textcolor{Bittersweet}{u_1},u^*_2),\dots,u^c_N(\textcolor{Bittersweet}{u_1},u^*_2)\right) $ and $\boldsymbol{u^{c*}}(u^*_1,\textcolor{Bittersweet}{u_2}) =$ $(u^c_1(u^*_1,\textcolor{Bittersweet}{u_2}),\dots,u^c_N(u^*_1,\textcolor{Bittersweet}{u_2}))$.
\end{itemize} 
\end{definition}

\section{Mean Field Game Formulation}
\label{sec:mfg_model}

	The present formulations correspond to the asymptotic regime whereby the number $N$ of minor players goes to $+\infty$. {Since players are identical, we focus on a \textit{representative} minor player.}
	
	\begin{remark}
	In the limit $N\to\infty$, {the representative player becomes infinitesimal; therefore, $\bar{u}_{-i}^{c}$ can be taken as $\bar u^c$.} Hereafter, $\bar{\mu}$ is used for the mean of the control of other minor players in the infinite number of player game instead of $\bar{u}_{-i}^{c}$ and it is equal to the mean of the controls of the all minor players, $\bar{u}^{c}$.
	\end{remark} 
	When we analyze the mean field game regime, \textit{representative} minor player's cost function can be written as:
	\begin{multline} \label{eq:minorcost_mfg}
        J^c(u^c;\bar \mu, u_1, u_2) =
		\mathbb{E}_{u_{0}^c \sim \mu_0}
		\Bigg\{
		\dfrac{\beta}{2} (u^{c}-u^{c}_{0})^{2} + 
		\dfrac{\eta}{2} (u^{c}-\bar{\mu})^{2}
		\\ - \left[(\alpha+u_{1})u^{c} + (\alpha+u_{2})(1-u^{c})-\dfrac{(u^{c})^{2}+(1-u^{c})^{2}}{2}\right]\Bigg\}
	\end{multline}
	where $u^c: [0,1] \rightarrow [0,1]$ is the feedback control function used by the representative minor player to update their initial preference rate $u_0^c \sim \mu_0$. Recall that we use the notation $\bar \mu$ for the mean of the control of the minor players. Hereinafter, the mean of the \textit{initial} preference rate is denoted as $\bar\mu_0=\mathbb{E}[u_0^c]$ or $\bar u_0^c$.
	
	The major players cost functions remain the same as in the case of $N$ finite. Only for consistency in the notation, $\bar{u}^c$ is changed to $\bar \mu$. We define the equilibrium notions in the mean field game model as follows:
	
	\begin{definition}[{Nash Equilibrium in the Mean Field Game with Multiple Major Players}] A strategy and a mean field tuple $(u^{c*},u_1^*, u_2^*, \bar \mu(u^{c*},u_1^*, u_2^*))$ form a \textit{Nash Equilibrium in the Mean Field Game regime with Multiple Major Players} if for any $u^c\in[0,1]$, $u_1, u_2 \in \mathbb{R}_+$ we have:
	\begin{equation*}
	\begin{aligned}
	    J^c(\textcolor{Bittersweet}{u^c};\bar \mu(\textcolor{black}{u^{c*}},u_1^*, u_2^*), u^*_1, u^*_2) &{\geq} J^c(u^{c*};\bar \mu(u^{c*},u_1^*, u_2^*), u^*_1, u^*_2)\\
	    J^1(\textcolor{Bittersweet}{u_1}; u^*_2, \bar \mu(u^{c*},\textcolor{Bittersweet}{u_1}, u_2^*)) &{\geq} J^1(u^*_1; u^*_2, \bar \mu(u^{c*},u_1^*, u_2^*))\\
	    J^2(\textcolor{Bittersweet}{u_2}; u^*_1, \bar \mu(u^{c*},u_1^*, \textcolor{Bittersweet}{u_2})) &{\geq} J^2(u^*_2; u^*_1, \bar \mu(u^{c*},u_1^*, u_2^*))
	\end{aligned}
	\end{equation*}
	
	\end{definition}
	
	\begin{definition}[{Multi Leader Follower Nash Equilibrium in the Mean Field Game}] A strategy and a mean field tuple $(u^{c*}(u_1^*, u_2^*),u_1^*, u_2^*, \bar \mu(u^{c*}(u^*_1, u_2^*)))$ form a \textit{Multi Leader Follower Nash Equilibrium in the Mean Field Game} regime if for any $u^c\in[0,1]$, $u_1, u_2 \in \mathbb{R}_+$, we have:
	\begin{equation*}
	\begin{aligned}
	    J^c(\textcolor{Bittersweet}{u^c}(u^*_1, u_2^*);\bar \mu(\textcolor{black}{u^{c*}}(u^*_1, u_2^*),u_1^*, u_2^*), u^*_1, u^*_2) &{\geq} J^c(u^{c*}(u^*_1, u_2^*);\bar \mu(u^{c*}(u^*_1, u_2^*),u_1^*, u_2^*), u^*_1, u^*_2)\\
	    J^1(\textcolor{Bittersweet}{u_1}; u^*_2, \bar \mu(u^{c*}(\textcolor{Bittersweet}{u_1}, u_2^*),\textcolor{Bittersweet}{u_1}, u_2^*)) &{\geq} J^1(u^*_1; u^*_2, \bar \mu(u^{c*}(u^*_1, u_2^*),u_1^*, u_2^*))\\
	    J^2(\textcolor{Bittersweet}{u_2}; u^*_1, \bar \mu(u^{c*}(u^*_1, \textcolor{Bittersweet}{u_2}),u_1^*, \textcolor{Bittersweet}{u_2})) &{\geq} J^2(u^*_2; u^*_1, \bar \mu(u^{c*}(u^*_1, u_2^*),u_1^*, u_2^*))
	\end{aligned}
	\end{equation*}
	 
	\end{definition}

	\section{Main Theoretical Results}
	\label{sec:theoretical_results}
	\subsection{Nash Equilibrium in the Mean Field Game with Major Players}
	 First, we focus on finding the Nash Equilibrium between major players and minor players. Here, all players are giving their best responses given other players' controls. We approach the model as follows:
	
	\begin{enumerate}
			\item First we fix the mean field $\bar{\mu}$ and solve 2-player game of Major Players to find their best responses given the mean field and the other major player's control:
			\begin{itemize}[label=$\rightarrow$]
				\item For major {player} 1, find $u_1^*=\varphi^1(u_2,\bar{\mu})$ $s.t:$
				\begin{equation}\label{Eq:NEMajor1Cost}
				u_1^* = \argmin_{u_{1} \in \mathbb{R_+}} \left\{ -\left(u_{1}\left(1-\bar{\mu}\right)-u_{2}\bar{\mu}\right)-\left(\dfrac{u_{1}+1}{u_{2}+1}\right) + \dfrac{c}{2}u_{1}^{2} \right\}.
				\end{equation}
				\item For major {player} 2, find $u_2^*=\varphi^2(u_1,\bar{\mu})$ $s.t:$
				\begin{equation}\label{Eq:NEMajor2Cost}
				u_2^*= \argmin_{u_{2} \in \mathbb{R_+}} \left\{ -\left(u_{2}\bar{\mu}-u_{1}\left(1-\bar{\mu}\right)\right)-\left(\dfrac{u_{2}+1}{u_{1}+1}\right) + \dfrac{c}{2}u_{2}^{2} \right\}.
				\end{equation}
			\end{itemize}
			\item We solve the 2-equation system of $u_1^*=\varphi^1(u_2,\bar{\mu})$ and $u_2^*=\varphi^2(u_1,\bar{\mu})$ to find the equilibrium controls $u_1^{**}=\phi^1(\bar{\mu})$ and $u_2^{**}=\phi^2(\bar{\mu})$ of major players in the 2-player game given the mean field of minor players.
			
			\item Then we fix mean field $ \bar{\mu} $, $u_1$ and $u_2$: 
			\begin{itemize}[label=$\rightarrow$]
				\item By considering the limit $N \rightarrow \infty$, solve the following mean field game problem for Minor Player where $ u_{0}^c $ is the initial control of minor players which is random, in other words:
				\begin{itemize}
				    \item Find $u^{c*}(u_{0}^c,\bar{\mu},u_1,u_2)$ s.t:
				\begin{multline}\label{eq:MinorCost}
				u^{c*} = \argmin_{u^{c}: [0,1]\rightarrow [0,1]} \mathbb{E}_{u_0^c \sim \mu_0}\Bigg\{\dfrac{1}{2} \left(u^{c}-u^{c}_{0}\right)^{2} + 
				\dfrac{1}{2} \left(u^{c}-\bar{\mu}\right)^{2} \\
				- \left[\left(\alpha+u_{1}\right)u^{c} + \left(\alpha+u_{2}\right)\left(1-u^{c}\right)-\dfrac{\left(u^{c}\right)^{2}+\left(1-u^{c}\right)^{2}}{2}\right]\Bigg\}.
				\end{multline}	
				\end{itemize}
				\item Fixed Point Argument:
				\begin{itemize}
				    \item Find $\bar{\mu}=\phi(u_1,u_2) \text{ s.t. } \bar{\mu} = \mathbb{E}[u^{c*}(u_0^c, \bar \mu, u_1, u_2)]$
				\end{itemize}
			\end{itemize}
			\item Solve the following 3-player system:
			\begin{align*}
			\begin{cases}
			u^{**}_1&=\phi^1(\bar{\mu}),\\
			u^{**}_2&=\phi^2(\bar{\mu}),\\      
			\bar{\mu}&=\phi(u_1^{**},u_2^{**}).
			\end{cases}
			\end{align*}
		\end{enumerate}
		
		 \textbf{Remark:} In step 2, instead of solving 2-player game and finding $u^{**}_1=\phi^1(\bar{\mu})$ and
		$u^{**}_2=\phi^2(\bar{\mu})$, we can continue directly to step 3. In this case, we would have the following 3-equation system at the end:
		\begin{align*}
		\begin{cases}
		u^*_1&=\varphi^1(u_2^*,\bar{\mu}),\\
		u^*_2&=\varphi^2(u_1^*,\bar{\mu}),\\      
		\bar{\mu}&=\phi(u_1^*,u_2^*).
		\end{cases}
		\end{align*}

	\begin{proposition} \label{prop:mfg_system}The final equation system is given as:
	\begin{subequations}
	\begin{gather}
	    u^{**}_{1}=-\frac{\bar{\mu}^2+2c\bar{\mu}-\bar{\mu}-c+c^2-\sqrt{(c+\bar{\mu})(c^2+5c-\bar{\mu}^2+\bar{\mu})(c-\bar{\mu}+1)}}{2c(c+\bar{\mu})}\label{ne_u1},\\
		u^{**}_{2}=\dfrac{\bar{\mu}-c+\sqrt{\frac{(c+\bar{\mu})(c^2+5c-\bar{\mu}^2+\bar{\mu})}{c-\bar{\mu}+1}}}{2c}\label{eq:ne_u2},\\
		\bar{\mu}=\dfrac{u_1^{**}-u^{**}_2+1+\mathbb{E}[u^{c}_0]}{3}\label{eq:ne_mu}.
	\end{gather}
	\end{subequations}
	\end{proposition}
	
	\begin{proof}[Proof of Proposition~\ref{prop:mfg_system}]
	\label{proof:mfg_system}
	The proof is consisted of 3 parts that are given above.
	\begin{enumerate}
	\item \textbf{Solution of 2-Major Player Game with Given Mean of Minor Player Control. } In this part, with the given mean of the minor players' controls, $\bar \mu$, we are analytically solving 2-player game of major players. For this reason, first we need to find best responses of major players, $\mathbb{R}_+ \ni u_{1}^*=\varphi^{1}(u_{2},\bar{\mu})$ and $\mathbb{R}_+ \ni u_{2}^*=\varphi^{2}(u_{1},\bar{\mu})$, that minimizes their cost functions.
	
	 \textbf{Remark:} For finding the controls that minimize the cost functions of major players, first order derivatives can be calculated. Although, since we deal with a constrained optimization, this minimizer may be out of the domain that the function is tried to be minimized. In this case, the minimizer would be on the boundary, this refers to case 2 in Figure~\ref{fig:DiffCaseMin}.
	
	 Moreover, when the cost functions of major players, \eqref{Eq:NEMajor1Cost} and \eqref{Eq:NEMajor2Cost} are checked, it can be seen that they are strictly convex in $u_1$ and $u_2$, respectively since it is assumed that $c > 0$. This means that we have unique minimizers. 
		
	\begin{figure}[t]
		\centering
		\begin{subfigure}[b]{0.4\linewidth}
			\includegraphics[width=\linewidth]{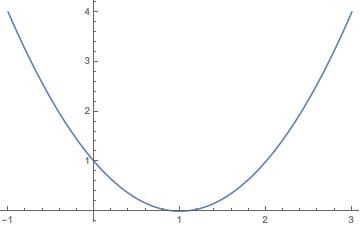}
			\caption{\textbf{Case 1:} Minimizer is in $\mathbb{R}_+$}
		\end{subfigure}
		\begin{subfigure}[b]{0.4\linewidth}
			\includegraphics[width=\linewidth]{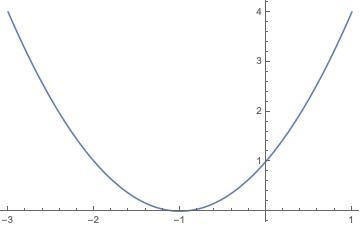}
			\caption{\textbf{Case 2:} Minimizer is out of $\mathbb{R}_+$}
		\end{subfigure}
		\caption{Different Cases for the Minimizers of Strictly Convex Functions}
		\label{fig:DiffCaseMin}
	\end{figure}
	 With above remark in our minds, first order conditions are calculated and minimizers are found as:
	\begin{align}
	u_{1}^*&=\dfrac{(1-\bar{\mu})+\frac{1}{u_{2}+1}}{c}\label{Map2Major1},\\
	u_{2}^*&=\dfrac{(\bar{\mu})+\frac{1}{u_{1}+1}}{c}\label{Map2Major2}.
	\end{align} 
	 In order to solve the 2-equation system, we plug $u_2$ into the equation of $u_1$ and have:
	\begin{equation}
	u_1^2 (c^2+c\bar{\mu})+u_1(c^2+2c\bar{\mu}-c-\bar{\mu}+\bar{\mu}^2)+(-2c-1+c\bar{\mu}+\bar{\mu}^2)=0.
	\end{equation}

	 The number of solutions of this equation depends on $\Delta = (c^2+2c\bar{\mu}-c-\bar{\mu}+\bar{\mu}^2)^2-4(c^2+c\bar{\mu})(-2c-1+c\bar{\mu}+\bar{\mu}^2)$. Since we have $\bar{\mu} \in [0,1]$ and $c>0$, it is concluded that $\Delta>0$ and we have 2 real-valued solutions. When they are analyzed further, it can be seen that in one of the solutions, $u_1, u_2 >0$ and in the other solution $u_1, u_2<0$. 
	
	 The set of positive solutions for $u^{**}_{1}=\phi^{1}(\bar{\mu})$ and $ u^{**}_{2}=\phi^{2}(\bar{\mu})$ are as following: 
	\begin{gather}
	u^{**}_{1}=-\frac{\bar{\mu}^2+2c\bar{\mu}-\bar{\mu}-c+c^2-\sqrt{(c+\bar{\mu})(c^2+5c-\bar{\mu}^2+\bar{\mu})(c-\bar{\mu}+1)}}{2c(c+\bar{\mu})}\label{eq:SolMajor1},\\
	u^{**}_{2}=\dfrac{\bar{\mu}-c+\sqrt{\frac{(c+\bar{\mu})(c^2+5c-\bar{\mu}^2+\bar{\mu})}{c-\bar{\mu}+1}}}{2c}\label{eq:SolMajor2}.    
	\end{gather}
		
	 In Figure~\ref{fig:ContMajMeth1}, plots of $u_1^{**}$ and $u^{**}_2$ under different	$\bar{\mu}$ and $c$ values can be found.
	\begin{figure}[h]
		\centering
		\begin{subfigure}[b]{0.9\linewidth}
			\includegraphics[width=\linewidth]{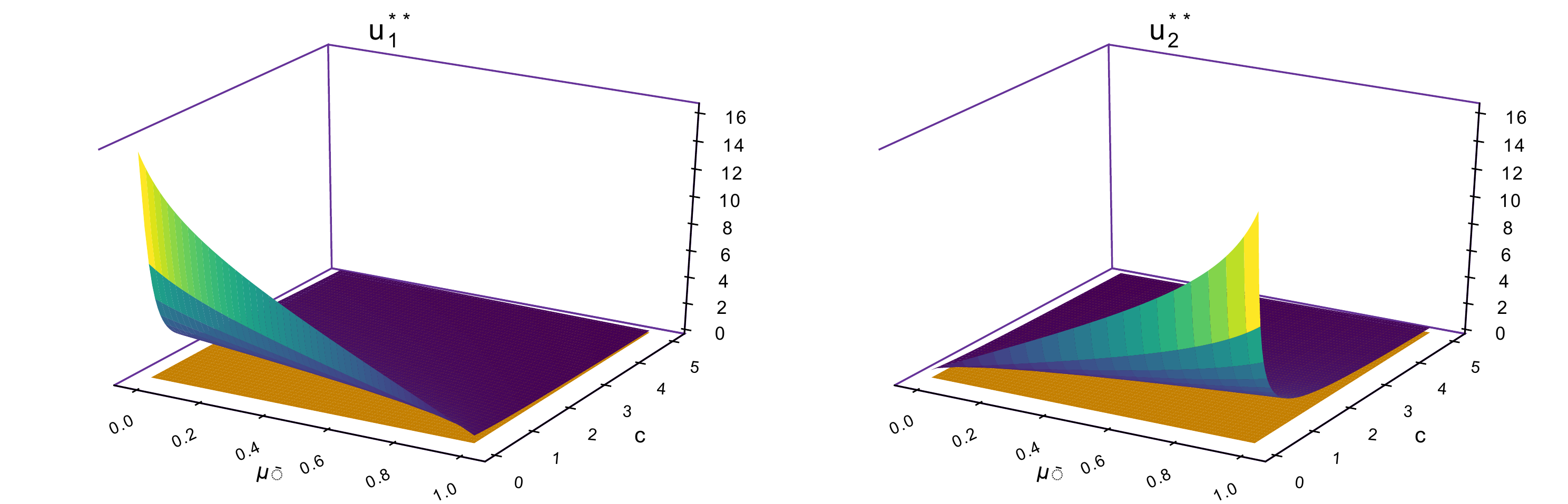}
		\end{subfigure}
		\caption{Control of Major Players under Different $\bar{\mu}$ and $c$ Values}
		\label{fig:ContMajMeth1}
	\end{figure}
	\item \textbf{Solution of the Constrained Optimization Problem of Minor Player. } In this section, with given controls for major players and the distribution of other minor players, we are solving the mean field game formulation of the optimization problem of minor players. Our goal is to find the mapping for the control of minor player that minimizes their cost function s.t.:
	\begin{equation}
	u^{c*} = \phi (\bar{\mu},u_1,u_2,u_{0}^c),
	\end{equation}
	where $u_1$ and $u_2$ denote the control of major players, $\bar{\mu}$ denotes the mean of the distribution of other minor players' control and random $ u_0^c \sim \mu_0 \in \mathcal{P}([0,1])$ denotes the initial preference rate of the minor player. Here, the control of minor player is a feedback function depending on their initial position, therefore in the cost function it is going to be denoted as $u^c(u_0^c)$.
	 The function that we want to minimize is as following:
            \begin{align*}
			J^c\left(u^c(.)\right)= & 
			\mathbb{E}_{u_0^c \sim \mu_0}\Bigg\{\dfrac{1}{2} \left(u^{c}(u_0^c)-u^{c}_{0}\right)^{2} +
			\dfrac{1}{2} \left(u^{c}(u_0^c)-\bar{\mu}\right)^{2} 
			- \Big[\left(\alpha+u_{1}\right)u^{c}(u_0^c) +\\
			&\left(\alpha+u_{2}\right)\left(1-u^{c}(u_0^c)\right)-\frac{\left(u^{c}(u_0^c)\right)^{2}+\left(1-u^{c}(u_0^c)\right)^{2}}{2}\Big]\Bigg\}\\
			= & \int_{0}^{1}\Bigg\{\dfrac{1}{2} \left(u^{c}(u_0^c)-u^{c}_{0}\right)^{2} +
			\dfrac{1}{2} \left(u^{c}(u_0^c)-\bar{\mu}\right)^{2} 
			- \Big[\left(\alpha+u_{1}\right)u^{c}(u_0^c) +\\
			&\left(\alpha+u_{2}\right)\left(1-u^{c}(u_0^c)\right)-\frac{\left(u^{c}(u_0^c)\right)^{2}+\left(1-u^{c}(u_0^c)\right)^{2}}{2}\Big]\Bigg\}d{\mu_0}(u^c_0).
	\end{align*}
	
	 Since the control of minor player is a feedback function, minimizing the integral can be done through minimizing the integrand. In other words, if we denote: 
\begin{align}\label{eq:MinorMinFunct}
g(u^c,u_0^c):=&\Bigg\{\dfrac{1}{2} \Big(u^{c}(u_0^c)-u^{c}_{0}\Big)^{2} + 
\dfrac{1}{2} \Big(u^{c}(u_0^c)-\bar{\mu}\Big)^{2} 
- \Big[\Big(\alpha+u_{1}\Big)u^{c}(u_0^c) +\\ &\Big(\alpha+u_{2}\Big)\Big(1-u^{c}(u_0^c)\Big)-\frac{\Big(u^{c}(u_0^c)\Big)^{2}+\Big(1-u^{c}(u_0^c)\Big)^{2}}{2}\Big]\Bigg\}.\nonumber
\end{align} 

Then:
\begin{align*}
\min_{u^c(u_0^c):[0,1]\rightarrow[0,1]} \int_{0}^{1} g(u^c(u_0^c),u_0^c)d{\mu_0}(u^c_0) \leftrightarrow \int_{0}^{1}\min_{u^c \in [0,1]} g(u^c,u_0^c)d\mu_0(u^c_0) 
\end{align*}	

	Now, we focus on minimizing \eqref{eq:MinorMinFunct}. Since this function is strictly convex in $u^c$, we may have three different cases for the minimizer as in Figure~\ref{DiffMinorOpt} given $u_0^c$. If the first order condition gives a minimizer that is smaller than 0 then the function is minimized at $u^c = 0$; on the other hand, if the first order condition gives a minimizer that is bigger than 1, then the function is minimized at $u^c = 1$, and in the other case, minimizer is found by first order condition.
	\begin{figure}
		\centering
		\includegraphics[width=0.5\linewidth]{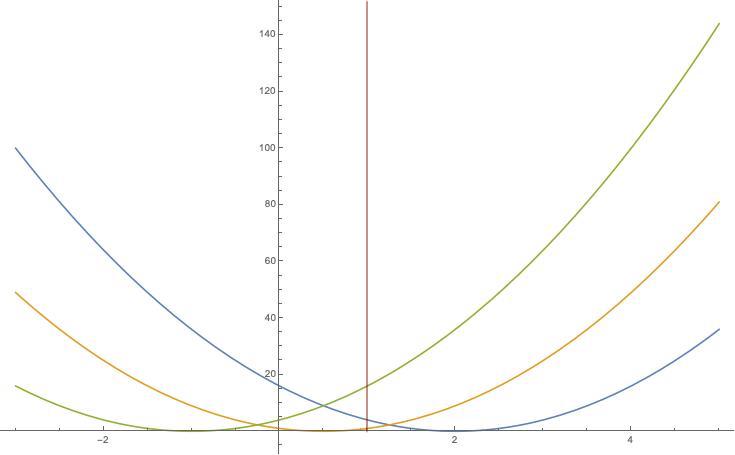}
		\caption{Different Possible Cases fo the Minimizers with Constraint $u_i^c \in [0,1]$}.
		\label{DiffMinorOpt}
	\end{figure}	

	When the first order condition is checked the minimizer is found to be:
	\begin{equation}
	\begin{aligned}\label{OptControlMinor}
	u^{c*} = \min\left(\max\left(0,\dfrac{\bar{\mu}+(u_1-u_2)+u^c_{0}+1}{4}\right),1\right). 
	\end{aligned}
	\end{equation}

	Now, we move to the fixed point argument part such that:
	\begin{equation}
		\bar{\mu} =\mathbb{E}\left[u^{c*}\right] = 	\mathbb{E}\left[\min\left(\max\left(0,\dfrac{\bar{\mu}+(u_1-u_2)+u^c_{0}+1}{4}\right),1\right)\right]. 
	\end{equation}

	Now, we assume that there exists $u_0^c$ such that $\frac{\bar \mu + (u_1-u_2) +u_0^c +1}{4} \notin [0,1]$. Further we define following sets:
\begin{equation*}
\begin{aligned}
    \mathcal A_< := \Big\{u_0^c : \dfrac{\bar \mu + (u_1-u_2) +u_0^c +1}{4}<0\Big\},\\
    \mathcal A_> := \Big\{u_0^c : \dfrac{\bar \mu + (u_1-u_2) +u_0^c +1}{4}>1\Big\},
    \end{aligned}
\end{equation*}
with the following probability measures.
\begin{equation*}
\begin{aligned}
    \mu_0\big[u_0^c \in \mathcal A_<\big]&=p_1,\qquad
    \mu_0\big[u_0^c \in \mathcal A_>\big]&=p_2.
    \end{aligned}
\end{equation*}

Intuitively this means that given $\bar \mu \in[0,1], u_1\in\mathbb{R}_+, u_2\in\mathbb{R}_+$, there may exist some $u_0^c$ that makes the expression $\frac{\bar \mu + (u_1-u_2) +\alpha +1}{4}$ smaller than 0 or bigger than 1 and their probability mass is given as $p_1\in[0,1]$ and $p_2\in[0,1]$, respectively. With this assumption the fixed point argument gives us:

\begin{equation}
\label{eq:mfg_fixed}
    \bar \mu = \frac{(1-p_1-p_2)(u_1-u_2 + \bar{u}_0^c +1) +4p_2}{4-(1-p_1-p_2)}.
\end{equation}

    \item \textbf{Solution of the System. } 
    
    \begin{lemma}
    \label{lem:fixed_mfg}
        For any given $u_0^c$, we have that  
        \begin{equation}
            0 \leq \dfrac{\bar{\mu}+(u_1-u_2)+u^c_{0}+1}{4} \leq1.
        \end{equation} In other words, we have $p_1=p_2=0$.
    \end{lemma}
    \begin{proof}[Proof of Lemma~\ref{lem:fixed_mfg}]
    First, we realize that any given $c$, the sign of $u_1-u_2$ depends on $\bar \mu$. If $\bar\mu<0.5(>0.5)$, we have $u_1-u_2>0(<0)$. Furthermore, if $\bar \mu=0.5$, we have $u_1-u_2=0$.
    
    \textbf{First, we assume that $\bar \mu=0.5$.} Then we have $0\leq\frac{\bar \mu +u_1-u_2+u_0^c+1}{4}\leq1$, by contradiction we find that $p_1=p_2=0$.
    
    \textbf{Secondly, we assume that $0\leq\bar \mu <0.5$.} In this case, we have $u_1-u_2>0$ and there does not exist $u_0^c\in[0,1]$ that gives $\frac{\bar{\mu}+u_1-u_2+u^c_{0}+1}{4}<0$. Therefore, we conclude that $p_1=0$. Now we focus on $p_2$. If $p_2\neq0$, it means that there exists some $u_0^c\in[0,1]$ that gives $\frac{\bar{\mu}+u_1-u_2+u^c_{0}+1}{4}>1$. From here we see that $u_1-u_2$ should be bigger than 1.5 and $x:=u_1-u_2+\bar u_0^c + 1 >2.5$. By using equation~\eqref{eq:mfg_fixed}, we have:
    \begin{equation*}
        \begin{aligned}
        &0\leq\frac{(1-p_2)(u_1-u_2 + \bar{u}_0^c +1) +4p_2}{4-(1-p_2)}<0.5\\
        &0\leq(1-p_2)(u_1-u_2 + \bar{u}_0^c +1) +4p_2<1.5+0.5p_2\\
        &0\leq (1-p_2)x + 4p_2<1.5+0.5p_2.
        \end{aligned}
    \end{equation*}
    Here since and $x>2.5$ and $1.5+0.5p_2<2$, we have a contradiction. Therefore, we also conclude that $p_2=0$. 
    
    \textbf{Finally, we assume that  $0.5<\bar \mu \leq1$.} In this case, we have $u_1-u_2<0$ and there does not exist $u_0^c\in[0,1]$ that gives $\frac{\bar{\mu}+u_1-u_2+u^c_{0}+1}{4}>1$. Therefore, we conclude that $p_2=0$. Now we focus on $p_1$. If $p_1\neq0$, it means that there exists some $u_0^c\in[0,1]$ that gives $\frac{\bar{\mu}+u_1-u_2+u^c_{0}+1}{4}<0$. From here we see that $u_1-u_2$ should be smaller than -1.5 and $x:=u_1-u_2+\bar u_0^c + 1 <0.5$.
    By using equation~\eqref{eq:mfg_fixed}, we have:
    \begin{equation*}
        \begin{aligned}
        0.5&<\frac{(1-p_1)(u_1-u_2 + \bar{u}_0^c +1)}{4-(1-p_1)}\leq1\\
        1.5+0.5p_1&<(1-p_1)(u_1-u_2 + \bar{u}_0^c +1) \leq3+p_1\\
        1.5+0.5p_1&< (1-p_1)x \leq 3+p_1.
        \end{aligned}
    \end{equation*}
    Since $c<0.5$ and $1.5+0.5p_1>1.5$, we have a contradiction. Therefore, we also conclude that $p_1=0$

    \end{proof}
    By using Lemma~\ref{lem:fixed_mfg}, we can calculate $\bar \mu$ and the final system becomes:
    
    \begin{gather}
		u^{**}_{1}=-\frac{\bar{\mu}^2+2c\bar{\mu}-\bar{\mu}-c+c^2-\sqrt{(c+\bar{\mu})(c^2+5c-\bar{\mu}^2+\bar{\mu})(c-\bar{\mu}+1)}}{2c(c+\bar{\mu})},\nonumber\\
		u^{**}_{2}=\dfrac{\bar{\mu}-c+\sqrt{\frac{(c+\bar{\mu})(c^2+5c-\bar{\mu}^2+\bar{\mu})}{c-\bar{\mu}+1}}}{2c},\label{eq:NESolSystem}\\
		\bar{\mu}=\dfrac{u_1^{**}-u_2^{**}+1+\bar u_0^c}{3}.		\nonumber
	\end{gather}
	 where the unit cost of advertisement, $c$, and the distribution of the initial control of minor players, $\bar u_0^c$, are given.

	\end{enumerate}

	\end{proof}
	
	\begin{theorem} \label{theorem:mfg_exist_uniq}
	There is a unique Nash Equilibrium in the Mean Field Game with Multiple Major Players.
	\end{theorem}
	\begin{proof}[Proof of Theorem~\ref{theorem:mfg_exist_uniq}]
	Since the system given in \eqref{eq:NESolSystem} gives the Nash Equilibrium solution in the Mean Field Game with Multiple Major Players it is enough to show that the system has a unique solution. If the 3-equation system in \eqref{eq:NESolSystem} has a solution, we can conclude that there is existence of the solution to this game. 
	
	The existence of the solution of this system can be showed by plugging in $u^{**}_1$ and $u^{**}_2$ values in the equation for $\bar \mu$. For any cost of advertisement $c$, we can see that for $\bar u_0^c\in [0, 0.5)$, there exists a $\bar \mu$ such that $\bar \mu \in [\frac{1+\bar u_0^c}{3}, 0.5]$ and for $\bar u_0^c\in (0.5, 1]$, there exists a $\bar \mu$ such that $\bar \mu \in [0.5, \frac{1+\bar u_0^c}{3}]$. Further we realize that if $\bar u_0^c=0.5$, we have $\bar \mu=0.5$ for any cost of advertisement. 
		
	\textcolor{black}{In order to show the uniqueness, we need to show that there is a unique fixed point for $\bar \mu$. Realize that after plugging in $u^{**}_1$ and $u^{**}_2$ values $\bar \mu$ can be found by solving the following equation:
	\begin{equation*}
	\begin{aligned}
	    f(\bar \mu) =\bar\mu-\frac{1}{3}\Big[&-\frac{\bar{\mu}^2+2c\bar{\mu}-\bar{\mu}-c+c^2-\sqrt{(c+\bar{\mu})(c^2+5c-\bar{\mu}^2+\bar{\mu})(c-\bar{\mu}+1)}}{2c(c+\bar{\mu})}\\
	    & -\dfrac{\bar{\mu}-c+\sqrt{\frac{(c+\bar{\mu})(c^2+5c-\bar{\mu}^2+\bar{\mu})}{c-\bar{\mu}+1}}}{2c}\Big]-\frac{1}{3}-\frac{\bar u_0^c}{3} =0
	\end{aligned}
	\end{equation*}	
	If we show that $f(\bar \mu)$ is strictly increasing or decreasing in $\bar \mu$ where $\bar \mu\in [0,1]$, the solution is unique. Therefore, we check the derivative of $f(\bar \mu)$:
	\begin{equation*}
	    \begin{aligned}
	    f^{\prime}(\bar \mu) = 1&+\frac{1}{3}\sqrt{(c^2+5c-\bar\mu^2+\mu)(c^2+c-\bar\mu^2+\mu)^{-1}}\\
	    &+\frac{(2\bar\mu-1)^2}{12c}\bigg((c^2+5c-\bar\mu^2+\mu)(c^2+c-\bar\mu^2+\mu)\bigg)^{-1/2}\\
	    &\times \bigg((c^2+5c-\bar\mu^2+\mu)(c^2+c-\bar\mu^2+\mu)^{-1}-1\bigg)
	    \end{aligned}
	\end{equation*}
	Since $c>0$ and $\bar\mu\in[0,1]$ we have $(c^2+5c-\bar\mu^2+\mu)>0$, $(c^2+c-\bar\mu^2+\mu)>0$ and $(c^2+5c-\bar\mu^2+\mu)(c^2+c-\bar\mu^2+\mu)^{-1}>1$. Therefore, $f(\bar\mu)$ is strictly increasing which concludes the uniqueness of the fixed point for $\bar \mu$. Since we previously showed that $u^{**}_1$ and $u^{**}_2$ are determined uniquely for any given $\bar \mu$, we conclude that we have a unique Nash equilibrium.} 
	
	\end{proof}

\subsection{Multi-Leader-Follower Nash Equilibrium in Mean Field Game with Major Players}
	In this setting, major players are playing a two-player game assuming that minor players are rational and constructing a Nash Equilibrium among themselves by taking into account the Mean Field Game Equilibrium of the minor players. Here, major players are in a sense Stackelbergian; however, we don't have a sequential game, major and minor players are giving their responses still simultaneously. At the end Multi-Leader-Multi-Follower Nash Equilibrium is found.

	\begin{enumerate}		
		\item We assume that major players think minor players are rational and constructing a Nash Equilibrium among themselves. We take major players' 2-player game equilibrium controls as $u_1$ and $u_2$.
		\item We fix $ \bar{\mu} $, and assume that minor players are giving the best response to the controls of major players:
		\begin{itemize}[label=$\rightarrow$]
			\item By considering the limit $N \rightarrow \infty$, we solve the following mean field game problem for the representative Minor Player, where $ u_{0}^c $ is the initial control of the minor player which is random:
			\begin{itemize}
			    \item We find $u^{c*}(u_0^c,\bar{\mu},u_1,u_2)$ $s.t:$		
			\begin{equation}
			\begin{aligned} \label{Eq:AlternatingMinorProblem}
			&u^{c*} = \argmin_{u^{c}: [0,1]\rightarrow[0,1]} \mathbb{E}_{u_0^c \sim \mu_0}\Bigg\{\dfrac{1}{2} \left(u^{c}-u^{c}_{0}\right)^{2} + 
			\dfrac{1}{2} \left(u^{c}-\bar{\mu}\right)^{2} \\
			&- \left[\left(\alpha+u_{1}\right)u^{c} + \left(\alpha+u_{2}\right)\left(1-u^{c}\right)-\dfrac{\left(u^{c}\right)^{2}+\left(1-u^{c}\right)^{2}}{2}\right]\Bigg\}.
			\end{aligned}
			\end{equation}			
			\end{itemize}
			\item We apply Fixed Point Argument: Find $\bar{\mu}(u_1,u_2) \textit{ s.t. } \bar{\mu} = \mathbb{E}[u^{c*}]$
		\end{itemize}

		\item Given the MFG equilibrium of minor players, $\bar{\mu}(u_1,u_2)$, we solve 2-player game of Major Players:
		\begin{itemize}[label=$\rightarrow$]
			\item First, we find best response of major player as a function of the other major player's control:
			\begin{itemize}
				\item For major {player} 1, find $u_1^*=\varphi^1(u_2)$ $s.t:$
				\begin{multline}
				\label{Eq:AlternatingMajor1Problem}
				u_1^* = \argmin_{u_{1} \in \mathbb{R_+}} \Bigg\{ -\left(u_{1}\left(1-\bar{\mu}\left(u_1,u_2\right)\right)-u_{2}\bar{\mu}\left(u_1,u_2\right)\right)-\left(\dfrac{u_{1}+1}{u_{2}+1}\right) + \dfrac{c}{2}u_{1}^{2} \Bigg\}.
				\end{multline}
				\item For major {player} 2, find $u_2^*=\varphi^2(u_1)$ $s.t:$
				\begin{multline}\label{Eq:AlternatingMajor2Problem}
				u_2^*= \argmin_{u_{2} \in \mathbb{R_+}} \Bigg\{ -\left(u_{2}\bar{\mu}\left(u_1,u_2\right)-u_{1}\left(1-\bar{\mu}\left(u_1,u_2\right)\right)\right)-\left(\dfrac{u_{2}+1}{u_{1}+1}\right) + \dfrac{c}{2}u_{2}^{2} \Bigg\}.
				\end{multline}
			\end{itemize}
			\item Then we solve the 2-equation system of $u_1^*=\varphi^1(u_2)$ and $u_2^*=\varphi^2(u_1)$ to find the equilibrium controls of major players in the 2-player game. 
		\end{itemize} 
		\item Finally we have a solution of the following form:	
		\begin{equation}
		\begin{aligned}\begin{cases}
		\bar{\mu} &= f_1\big(\mathbb{E}[u_{0}^c]\big),\\
		u_1^{**} &= f_2\big(\mathbb{E}[u_{0}^c],c\big),\\
		u_2^{**} &= f_3\big(\mathbb{E}[u_{0}^c],c\big).\end{cases}
		\end{aligned} 
		\end{equation}
	\end{enumerate}
	
	\begin{proposition} 
	\label{prop:mlfne}
	The final system for the Multi-Leader-Follower Nash Equilibrium in our model is given as
	\begin{equation}
    \begin{aligned}\label{eq:mlfne_system}
    u_1^{**} =& \dfrac{1}{3+3c+\mathbb{E}[u^{c}_0]}\Bigg\{1-2\mathbb{E}[u^{c}_0]+\dfrac{(1+3c-\mathbb{E}[u^{c}_0]-\Delta)(\mathbb{E}[u^{c}_0]-3c-4)}{2(2+3c)}\Bigg\},\\
    \\ 
    u_2^{**} =& \dfrac{-1-3c+\mathbb{E}[u^{c}_0]+\Delta}{2(2+3c)},
    \\
    \bar{\mu} =& \dfrac{1}{3}\Bigg\{ \dfrac{1}{3+3c+\mathbb{E}[u^{c}_0]}\Big\{1-2\mathbb{E}[u^{c}_0]+\dfrac{(1+3c-\mathbb{E}[u^{c}_0]-\Delta)(\mathbb{E}[u^{c}_0]-3c-4)}{2(2+3c)}\Big\},\\
    &\qquad\qquad\qquad\qquad\qquad\qquad\qquad\ -\dfrac{-1-3c+\mathbb{E}[u^{c}_0]+\Delta}{2(2+3c)}+1+\mathbb{E}[u^{c}_0]\Bigg\},
    \end{aligned} 
    \end{equation}
    where
    $$
    \Delta:=\sqrt{\dfrac{(3+3c+\mathbb{E}[u^{c}_0])(36+57c+9c^2+\mathbb{E}[u^{c}_0]-\mathbb{E}[u^{c}_0]^2)}{4+3c-\mathbb{E}[u^{c}_0]}}.
    $$
	\end{proposition}
	
	\begin{proof}[Proof of Proposition~\ref{prop:mlfne}]
		Here, the Mean Field Game solution of the minor player stays the same, remember that we have the following minimizer for the cost function of the minor player:
    \begin{equation}\label{OptControlMinor_2}
    \begin{aligned}
    u^{c*} = \min\left(\max\Big(0,\dfrac{\bar{\mu}+(u_1-u_2)+u^c_{0}+1}{4}\Big),1\right) .
    \end{aligned}
    \end{equation}
    By using the fixed point argument we find
    \begin{equation}
    \label{eq:mlf_fixed}
        \bar \mu(u_1,u_2) = \frac{(1-p_1-p_2)(u_1-u_2 + \bar{u}_0^c +1) +4p_2}{4-(1-p_1-p_2)},
    \end{equation}
    where $p_1$ and $p_2$ are defined as in proof of Proposition~\ref{prop:mfg_system}. 
    
    After, this $\bar{\mu}$ is plugged into the cost functions of major players, firstly it is checked whether the cost functions of major players are strictly convex in their own controls. Since for the Second Order Condition we have SOC$=\frac{2(1-p_1-p_2)}{4-(1-p_1-p_2)}+c>0$, it is concluded that, cost function of major player 1 \eqref{Eq:AlternatingMajor1Problem} is strictly convex in $u_1$ and cost function of major player 2 \eqref{Eq:AlternatingMajor2Problem} is strictly convex in $u_2$. After first order condition is checked, it is concluded that we have the following minimizers:
	\begin{equation}
	\begin{aligned}
	u_1^* &= \left.\Big[\frac{1}{u_2+1}+\frac{4(1-p_2)-(1-p_1-p_2)(2+\bar u^c_0)}{4-(1-p_1-p_2)}\Big]\middle/\Big[c+\frac{2(1-p_1-p_2)}{4-(1-p_1-p_2)}\Big],\right.\\[3mm] 
	u_2^* &= \left.\Big[\frac{1}{u_1+1}+\frac{4p_2+(1-p_1-p_2)(1+\bar u^c_0)}{4-(1-p_1-p_2)}\Big]\middle/\Big[c+\frac{2(1-p_1-p_2)}{4-(1-p_1-p_2)}\Big].\right.
	\end{aligned} 
	\end{equation}
    
    \begin{lemma}\label{lemma:inequality_mlf}
    For any given $u_0^c$, we have that  
        \begin{equation}
            0 \leq \dfrac{\bar{\mu}+(u_1-u_2)+u^c_{0}+1}{4} \leq1.
        \end{equation} In other words, we have $p_1=p_2=0$.
    \end{lemma}
    
\begin{proof}[Proof of Lemma~\ref{lemma:inequality_mlf}]
  First we realize that since $u_0^c\in[0,1]$, given $\bar \mu, u_1,u_2$, we cannot have $p_1\neq0$ and $p_2\neq0$ simultaneously. Therefore, we look at two cases and show contradictions in these cases.
  
  \textbf{First we assume that} $p_1\neq0, p_2=0$. In this case, when the two player system of major players in \ref{Eq:2EqSystMeth2} is solved, we realize given any $p_1>0, c>0$ and $\mathbb{E}[u_0^c]$, we have $u_1-u_2>-1$. Since $\bar \mu \in[0,1]$, we conclude that $\bar{\mu}+(u_1-u_2)+u^c_{0}+1>0$ for all $u_0^c\in[0,1]$; therefore, we conclude that $p_1=0$ which contradicts with our initial assumption. 
  
    \textbf{Secondly, we assume that} $p_1=0, p_2\neq0$. In this case, when the two player system of major players in \ref{Eq:2EqSystMeth2} is solved, we realize given any $p_1>0, c>0$ and $\mathbb{E}[u_0^c]$, we have $u_1-u_2<1$. Since $\bar \mu \in[0,1]$, we conclude that $\bar{\mu}+(u_1-u_2)+u^c_{0}+1<4$ for all $u_0^c\in[0,1]$; therefore, we conclude that $p_2=0$ which contradicts with our initial assumption. 
\end{proof}
    
	By using Lemma~\ref{lemma:inequality_mlf}, we conclude that:	
	\begin{equation}
	\begin{aligned}
	\bar{\mu}(u_1,u_2)=\dfrac{u_1-u_2+1+\mathbb{E}[u^{c}_0]}{3}.
	\end{aligned}
	\end{equation}
	Further again by using Lemma~\ref{lemma:inequality_mlf}, we can rewrite the minimizers for the major players' cost functions as follows:
	\begin{equation}
	\begin{aligned}\label{Eq:2EqSystMeth2}
	u_1^* = \frac{2u_2-\mathbb{E}[u^{c}_0]u_2-\mathbb{E}[u^{c}_0]+5}{3cu_2+3c+2u_2+2},\\[3mm]
	u_2^* = \frac{u_1+\mathbb{E}[u^{c}_0]u_1+\mathbb{E}[u^{c}_0]+4}{3cu_1+3c+2u_1+2}.
	\end{aligned} 
	\end{equation}

     Now, the 2-equation system \eqref{Eq:2EqSystMeth2} needs to be solved in order to find the equilibrium. When the solutions are checked, we saw that there exist 2 sets of solutions; one positive and one negative set. Because of the nonnegativity assumption in the controls of the major players, we conclude that the positive set gives the unique optimal control for the major players.
     	\end{proof}

	\begin{theorem} 
	\label{the:mlfne_existence}
	There exists a unique Multi-Leader-Follower Nash Equilibrium.
	\end{theorem}

	\begin{proof}[Proof of Theorem~\ref{the:mlfne_existence}]
		Since the system given in \eqref{eq:mlfne_system} gives the Multi-Leader-Follower Nash Equilibrium solution and since this system has a unique solution, there exists a unique Multi-Leader-Follower Nash Equilibrium in the Mean Field Game with Major Players.    
	\end{proof}
	
\section{Experiment Results}
\label{sec:Numerical_results}
	\subsection{NE: Experiments and Intuitive Remarks} 
	\label{subsec:NE_remarks}

	The solution of the above equation system (\ref{eq:NESolSystem}), $\bar{\mu}$, $u_1^{**}$ and $u_2^{**}$, are given in the plots (Figure \ref{fig:BarMinMeth1}, \ref{fig:BarMaj1Meth1}, and \ref{fig:BarMaj2Meth1}).

	First we interpret the results related to the market shares. (Figure~\ref{fig:BarMinMeth1}) From the results, we can infer that at any level of unit cost of advertisement, $c$, and initial market share (mean of initial control of minor players), $\mathbb{E}[u_0^c]$, market shares of companies become closer. In other words $\bar{\mu}$ is closer to 0.5 than $\mathbb{E}[u_0^c]$. Secondly, we can see that this effect is higher when unit cost of advertisement is lower. This is because at lower levels of $c$, both companies are advertising heavily and they affect the marginalized customers on each end. As unit cost of advertisement increases, this effect is less prevalent and market share of product 1 settles down at $\frac{1+\mathbb{E}[u_0^c]}{3}$. On the other hand, if the cost of advertisement, $c$, goes to 0, market becomes perfectly shared, in other words $\bar{\mu} \rightarrow 0.5$ as $c \searrow 0$. Finally, we infer that if at the beginning, the market starts perfectly shared (ie. $\mathbb{E}[u_0^c]=0.5$) it stays in that way (ie. $\bar{\mu}=0.5$).
	
	Now we can focus on the interpretation of the results related to the equilibrium advertisement efficiencies (Figure \ref{fig:BarMaj1Meth1} and \ref{fig:BarMaj2Meth1}). As expected at any initial market share ($\mathbb{E}[u_{0}^c]$) level, as the unit cost of advertisement increases, major players are advertising less because of the high costs. In other words, when $c \nearrow \infty$, $u_1^{**}$ and $u_2^{**}$ goes to 0. 
	
	Secondly, since the cost functions of major players are the same, symmetric behaviour is observed. In other words, at the same level of unit cost of advertisement, companies are advertising the same amount if they start at the same level of market share. Following this, if the market is shared perfectly initially ($\mathbb{E}[u_{0}^c]=0.5$), both major players have the same advertisement efficiency level; in other words at the equilibrium.
	
	\begin{figure}[H]
    \centering
	\begin{subfigure}{\linewidth}
	\centering
		\includegraphics[width=0.9\linewidth]{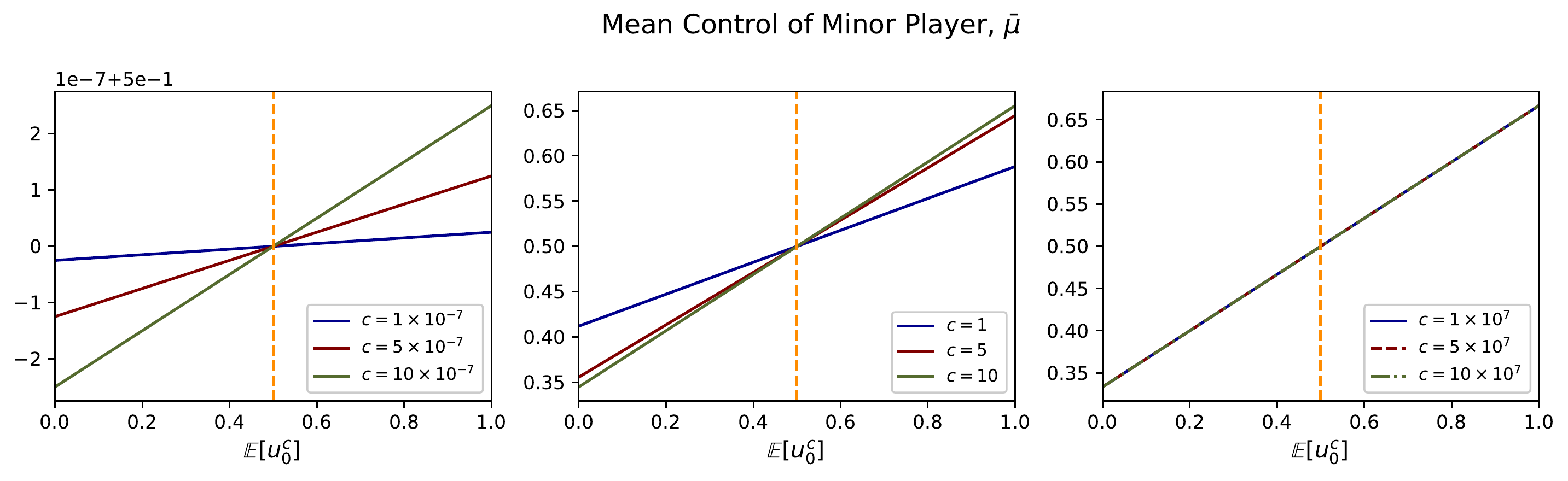}
	\end{subfigure}
		\begin{subfigure}{\linewidth}
		\centering
		\includegraphics[width=0.95\linewidth]{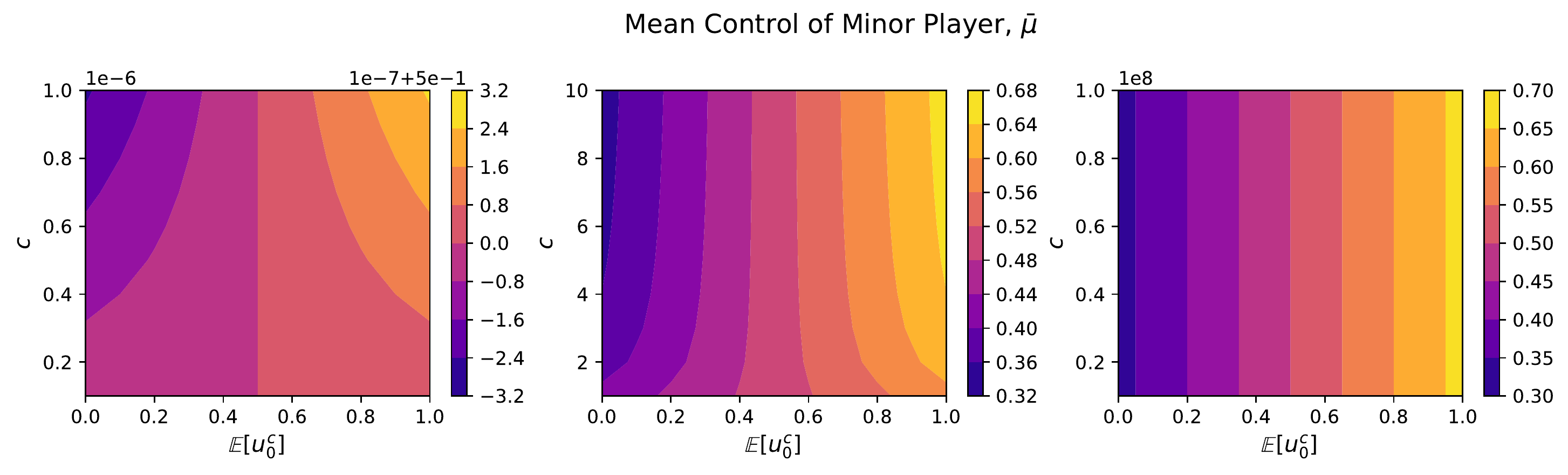}
	\end{subfigure}
	\caption{\textbf{NE:} $\bar{\mu}$, under Different  $\mathbb{E}[u_0^c]$ and $c$ Values}
	\label{fig:BarMinMeth1}
\end{figure}

		\begin{figure}[H]
	\centering
	\begin{subfigure}{\linewidth}
	\centering
		\includegraphics[width=0.9\linewidth]{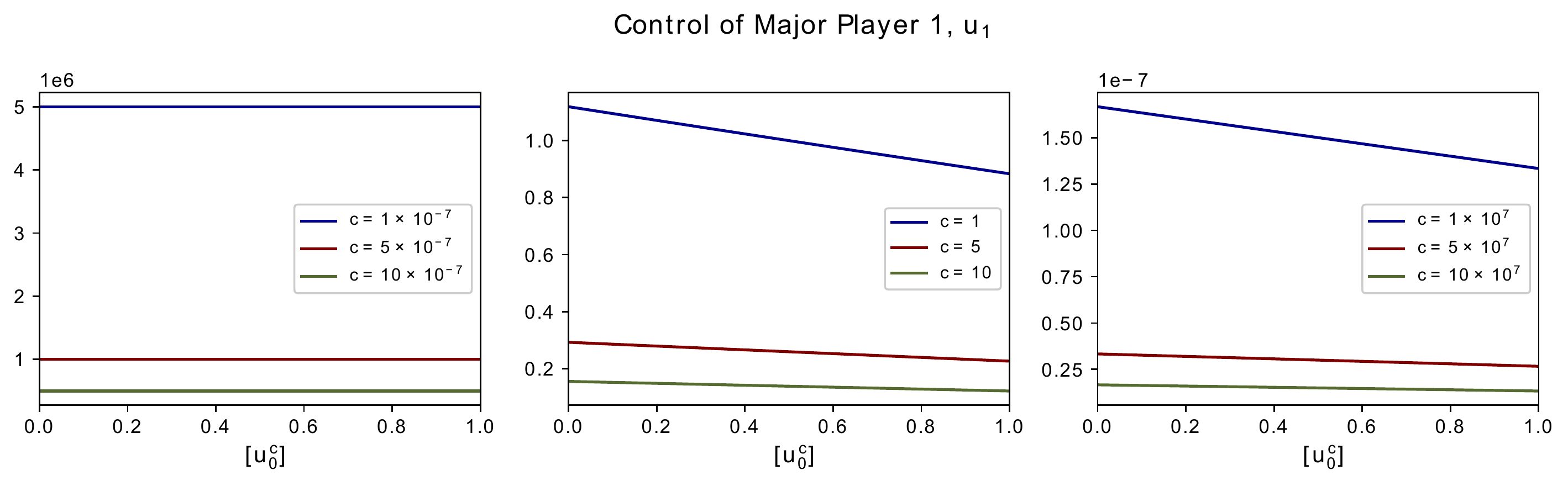}
	\end{subfigure}
		\begin{subfigure}{\linewidth}
		\centering
		\includegraphics[width=0.95\linewidth]{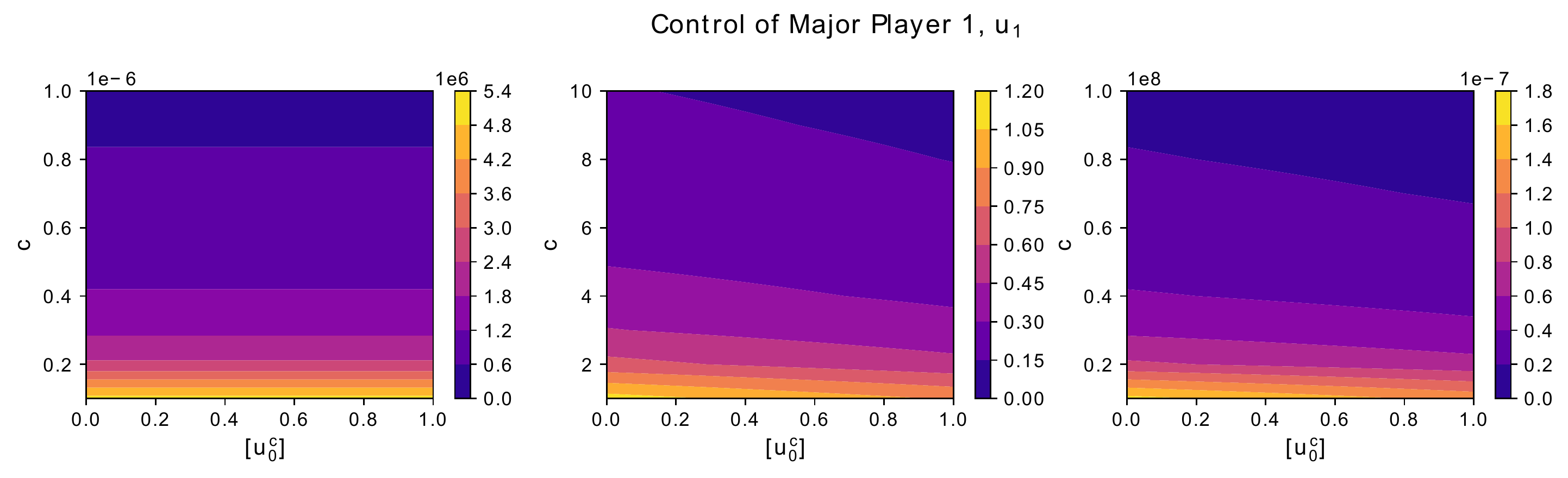}
	\end{subfigure}
	\caption{\textbf{NE:} Control of Major Player 1, $u_1^{**}$, under Different  $\mathbb{E}[u_0^c]$ and $c$ Values}
	\label{fig:BarMaj1Meth1}
\end{figure}

	\begin{figure}[H]
	\centering
	\begin{subfigure}{\linewidth}
	\centering
		\includegraphics[width=0.9\linewidth]{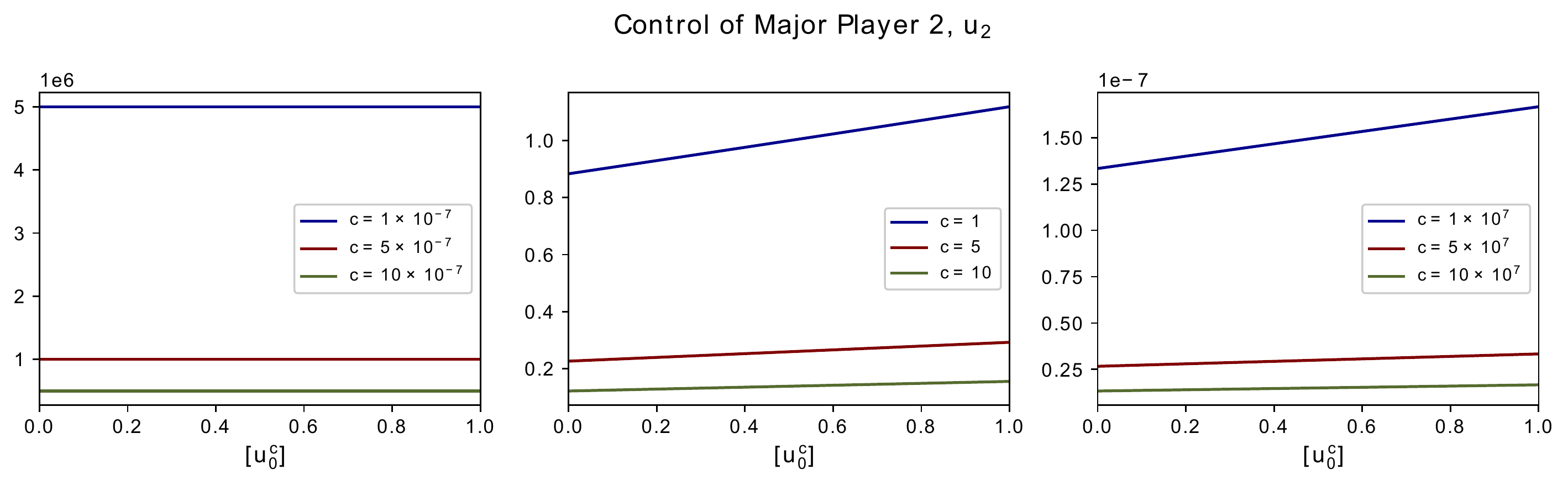}
	\end{subfigure}
		\begin{subfigure}{\linewidth}
		\centering
		\includegraphics[width=0.95\linewidth]{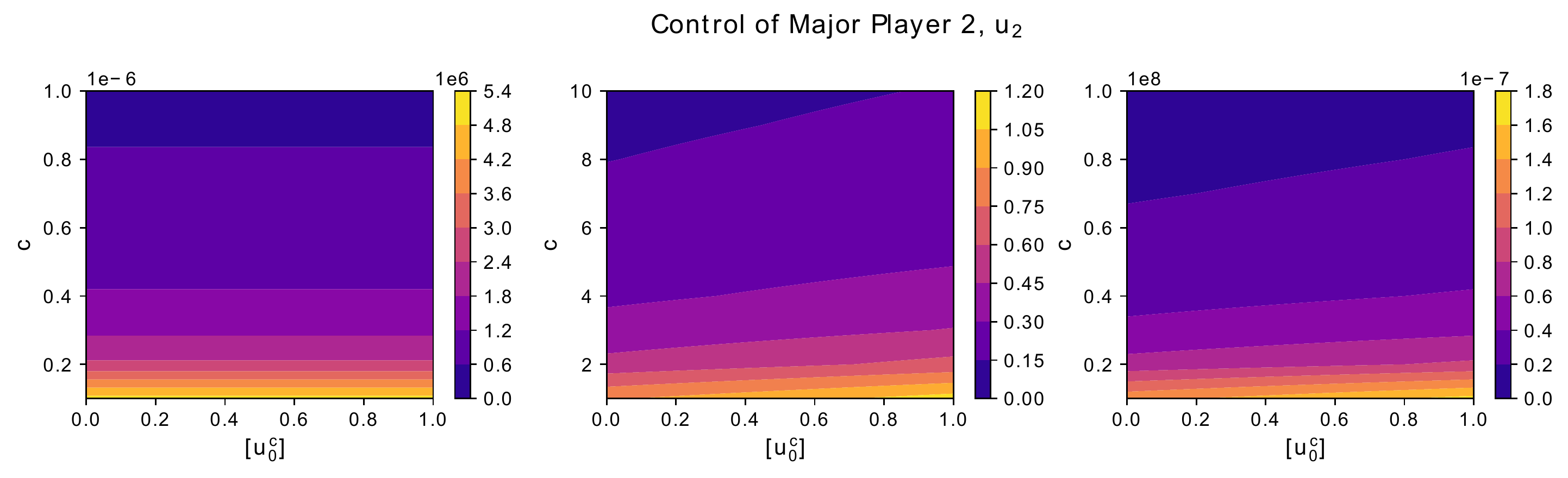}
	\end{subfigure}
	\caption{\textbf{NE:} Control of Major Player 2, $u_2^{**}$, under Different  $\mathbb{E}[u_0^c]$ and $c$ Values}
	\label{fig:BarMaj2Meth1}
\end{figure}
	
	 Furthermore, we realize that at every cost of unit advertisement, as initial market share of product 1 ($\mathbb{E}[u_{0}^c]$) increases, $u_1^{**}$ decreases slightly and $u_2^{**}$ increases slightly. The explanation as follows: Since the market share of Product 1 is increasing with $\mathbb{E}[u_{0}^c]$, Company 1 decides on smaller advertisement efficiency level when its market share is larger; in other words, $u_1^{**}$ gets smaller. This result is in the same direction with the findings in \cite{Doyle1968}. On the other hand, since the initial market share for Product 2 is smaller, Company 2 decides to have larger advertisement efficiency level, in order to be able to persuade customers of the opposing company.

    Finally, we note that there is \textbf{no} Nash Equilibrium, where market share becomes polarized. In other words, for any $u_1$ and $u_2$, having $\bar{\mu}=0$ or $\bar{\mu}=1$ is not a Nash equilibrium. Further, there is \textbf{no} Nash Equilibrium, where companies are not advertising. In other words, for any $\bar{\mu}$ having $u_1=0$ and/or $u_2=0$ is not a Nash Equilibrium.

	\subsection{MLF-NE:Experiments and Intuitive Remarks} 		
	\label{subsec:MLFNE_remarks}

	\begin{figure}[H]
    \centering
		\begin{subfigure}[c]{\linewidth}
		\centering
		\includegraphics[width=0.95\linewidth]{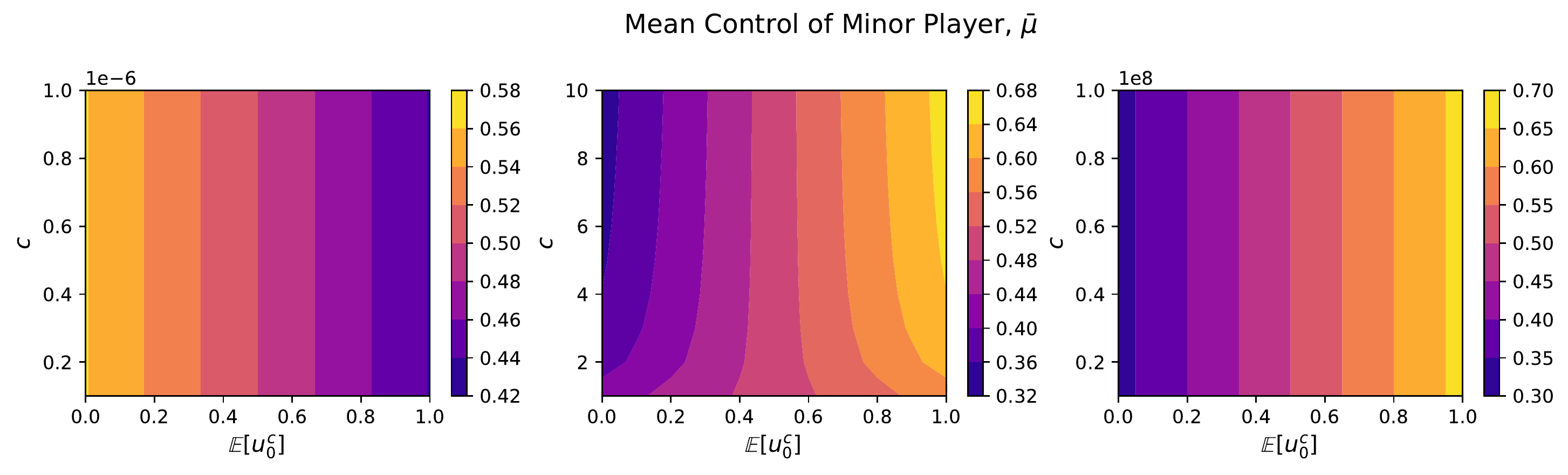}
	\end{subfigure}
	\caption{\textbf{MLF-NE:} $\bar{\mu}$, under Different  $\mathbb{E}[u_0^c]$ and $c$ Values}
	\label{fig:ContMinMeth2}
\end{figure}

\begin{figure}[H]
	\centering
		\begin{subfigure}[b]{\linewidth}
		\centering
		\includegraphics[width=0.95\linewidth]{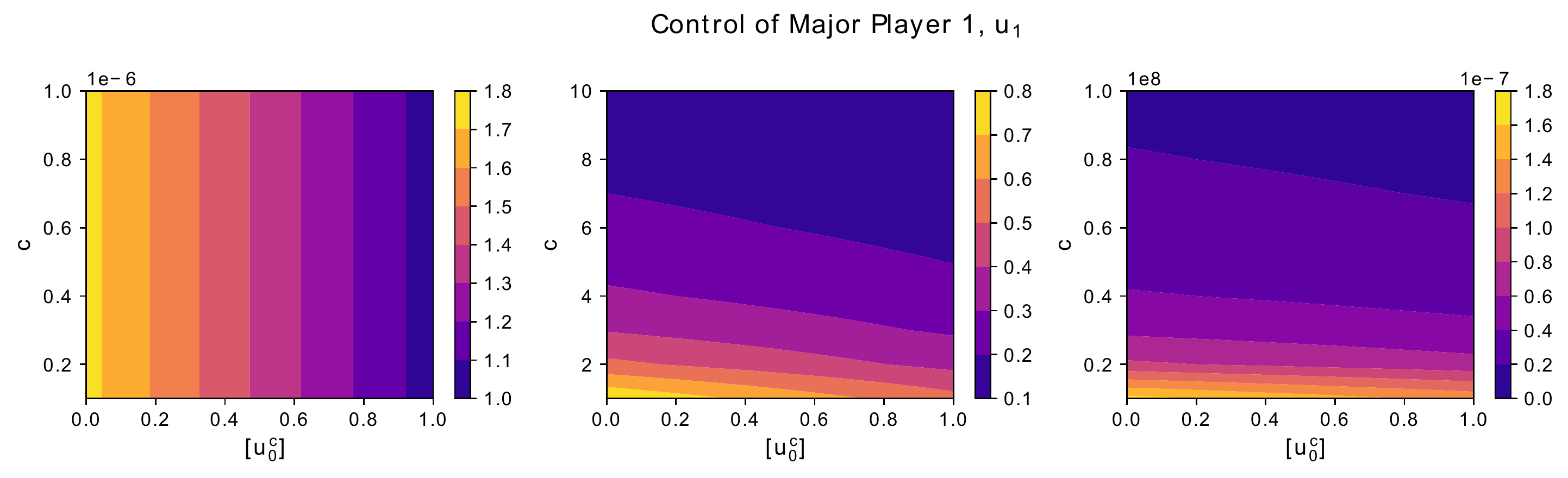}
	\end{subfigure}
	\caption{\textbf{MLF-NE:} Control of Major Player 1, $u_1^{**}$, under Different  $\mathbb{E}[u_0^c]$ and $c$ Values}
	\label{fig:ContMajMeth2u1}
\end{figure}

	In this section, we interpret the results of the Multi-Leader-Follower Nash Equilibrium. (Figure~\ref{fig:ContMinMeth2}, \ref{fig:ContMajMeth2u1} and \ref{fig:ContMajMeth2u2}) First, we note that the same intuitive remarks given in Subsection~\ref{subsec:NE_remarks} hold for the results. We can summarize them briefly here: 1) The market becomes more homogenized at the equilibrium and this homogenization effect is more prevalent when the cost of advertisement is lower. 2) The behavior of companies are symmetric. If the initial market shares are the same, they advertise the same amounts. 3) At every cost of unit advertisement, as the initial market share of a company increases, its advertising level decreases. 4) There is \textbf{no} Multi-Leader-Follower Nash Equilibrium, where market shares become polarized or where companies are not advertising.

\begin{figure}[H]
	\centering
		\begin{subfigure}[b]{\linewidth}
		\centering
		\includegraphics[width=0.95\linewidth]{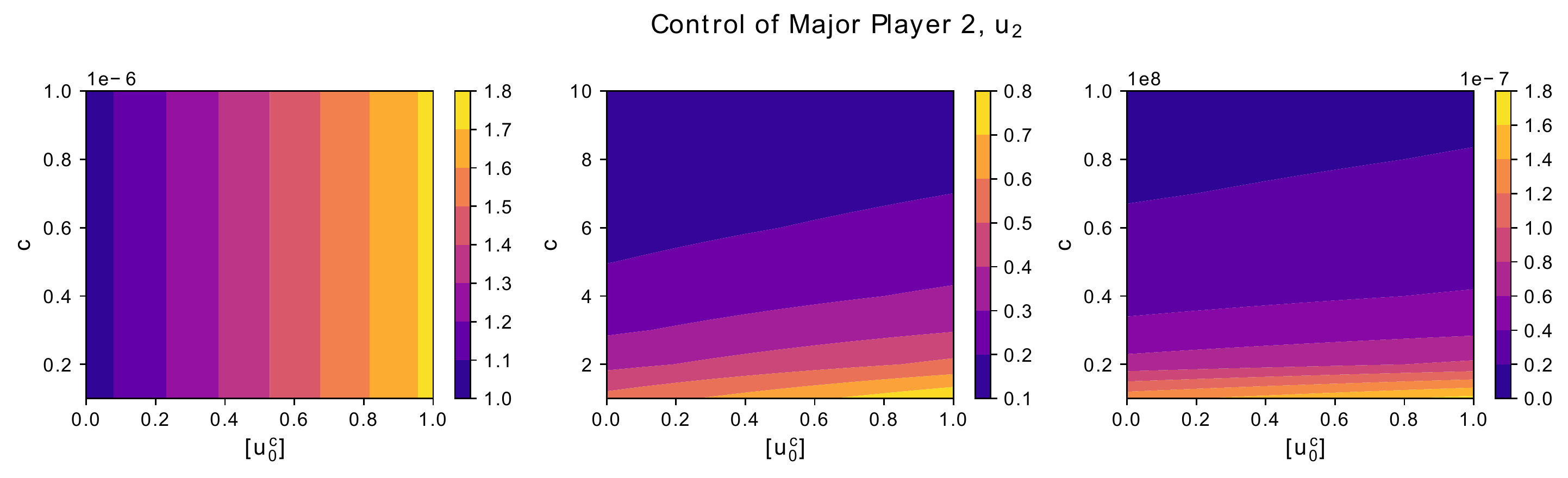}
	\end{subfigure}
	\caption{\textbf{MLF-NE:} Control of Major Player 2, $u_2^{**}$, under Different  $\mathbb{E}[u_0^c]$ and $c$ Values}
	\label{fig:ContMajMeth2u2}
\end{figure}

\section{Comparison of Nash Equilibrium and Multi-Leader-Follower Nash Equilibrium}
We analyzed the different equilibrium behavior (NE \& MLF-NE) under different initial market share and cost of advertisement choices in Subsections~\ref{subsec:NE_remarks} and \ref{subsec:MLFNE_remarks}. In this part of the report, differences and similarities between standard Nash Equilibrium \textbf{(NE)} and the Multi-Leader-Follower Nash Equilibrium \textbf{(MLF-NE)} are going to be stated. Firstly, we need to emphasize again that the Multi-Leader-Follower Nash equilibrium does not give the Nash Equilibrium among all the players in the classical sense. It gives an Stackelbergian-like equilibrium; therefore, we end up with different results.

 As stated above in Subsections~\ref{subsec:NE_remarks} and \ref{subsec:NE_remarks}, both equilibrium points have some similarities. For example, major players' advertisement efficiency decisions are moving in the same direction in both of the equilibrium notions under the changes in unit cost of advertisement, $c$, or initial market share, $\mathbb{E}[u_0^c]$. For example, in both cases, major players are advertising less when $c$ is higher. Moreover, in both of the equilibrium notions, market shares become closer and if the unit cost of advertisement goes to infinity, $\bar{\mu}$ goes to $\frac{1+\mathbb{E}[u_0^c]}{3}$. 

\begin{figure}[H]
    \centering
		\includegraphics[width=1\linewidth]
			{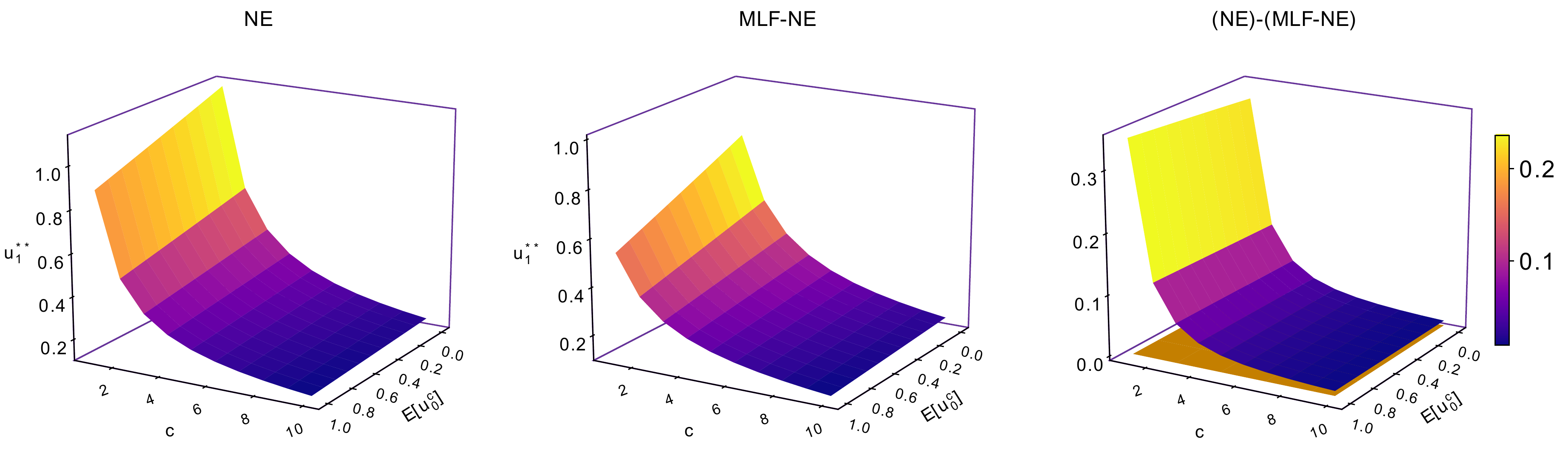}
			\caption{{\textbf{NE vs. MLF-NE:} Control of Major Player 1, $u_1^{**}$, under Different $\mathbb{E}[u_0^c]$ and $c$ Values in NE (left) and in MLF-NE (middle); The difference of the control of the Major Player 1, $u_1^{**}$, in NE vs MLF-NE (right).}}
			\label{fig:Major1_comp_3D}
	\end{figure}

    \begin{figure}[H]
    \centering
			\includegraphics[width=0.65\linewidth, height=0.45\linewidth]
			{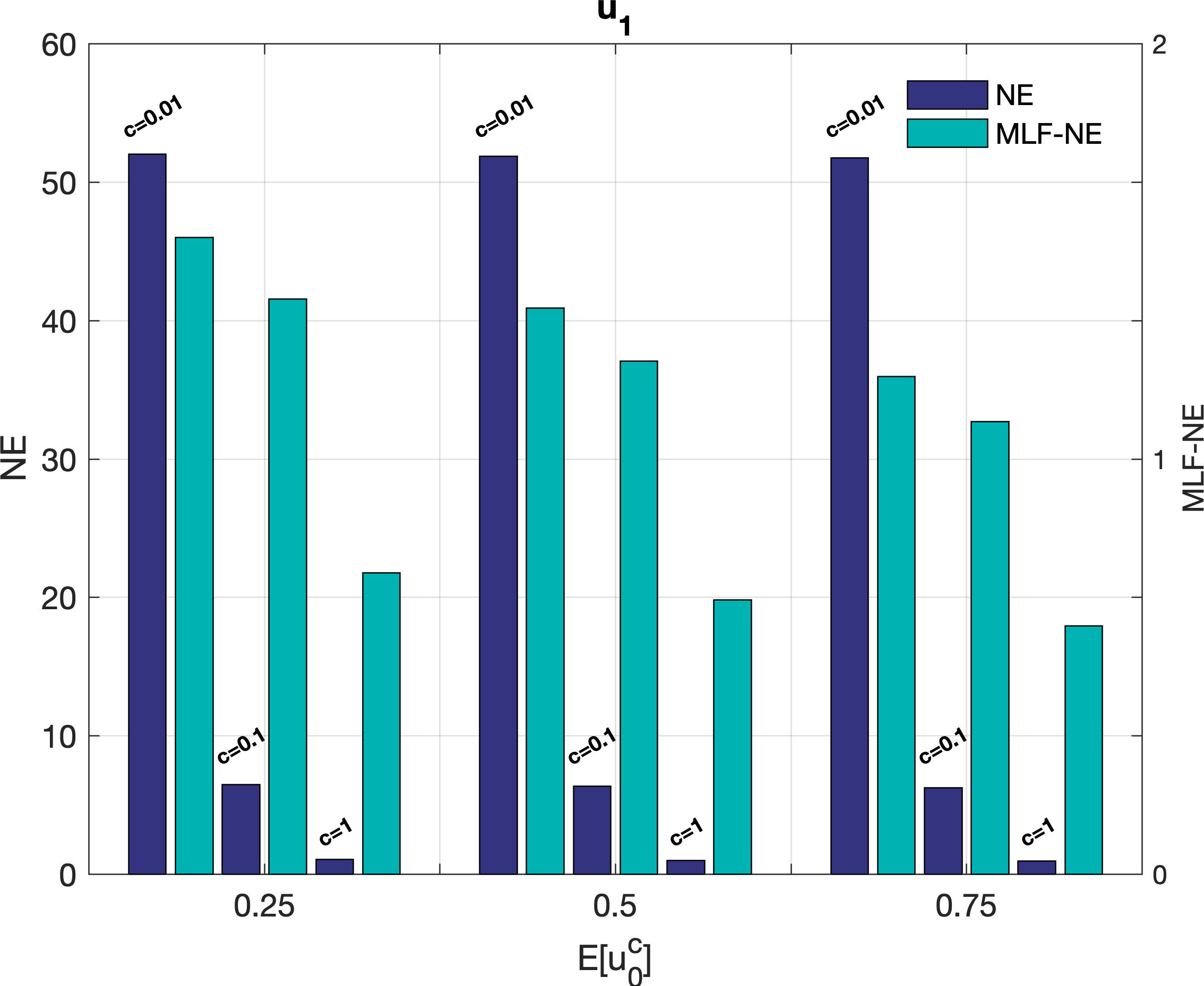}
			\caption{\textbf{NE vs. MLF-NE:} Control of Major Player 1, $u_1^{**}$, under Different $\mathbb{E}[u_0^c]$ and $c$ Values}
			\label{fig:Major1_comp}
	\end{figure}

	\begin{figure}[H]
	\centering
			\includegraphics[width=0.65\linewidth, height=0.45\linewidth]{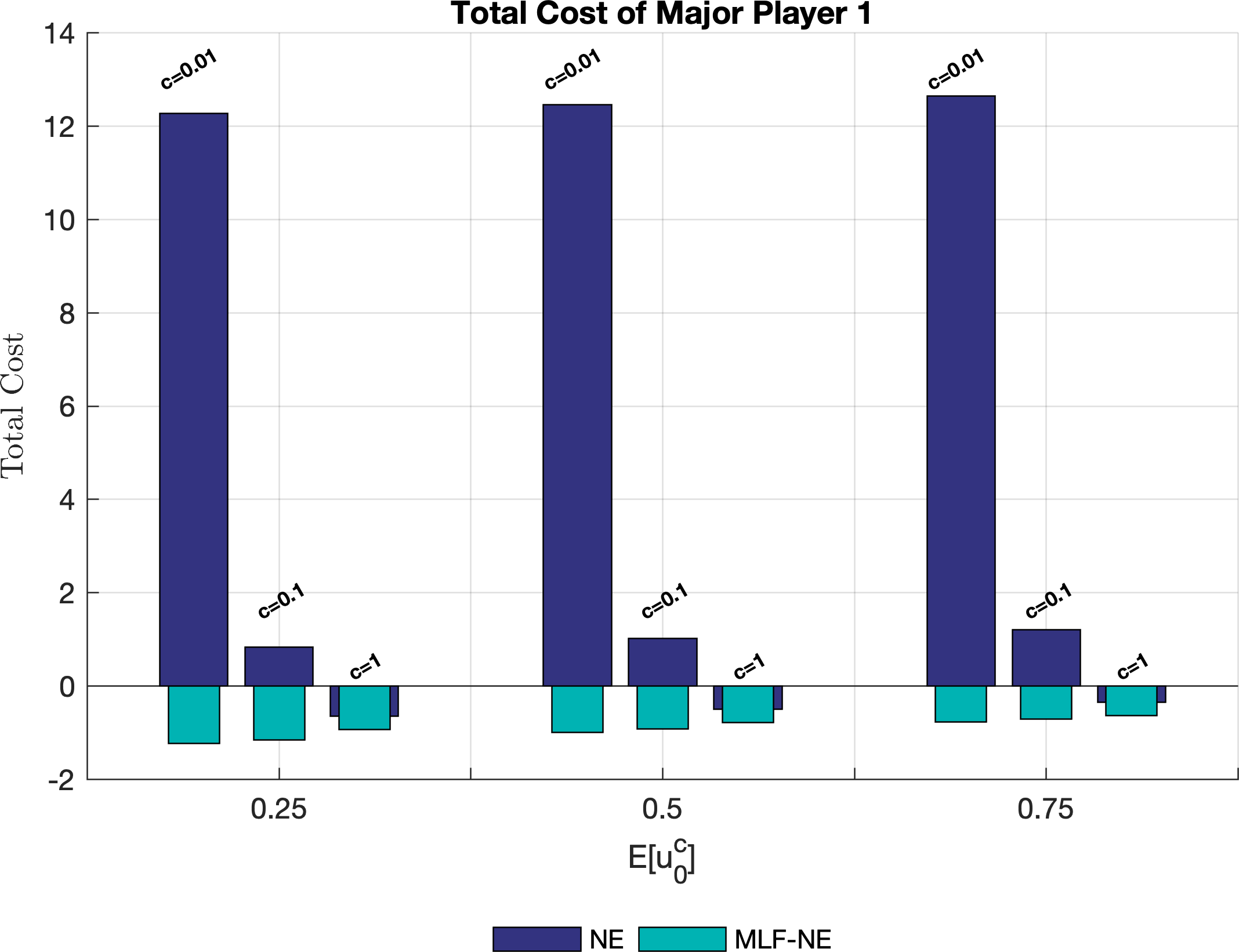}
			\caption{\textbf{NE vs. MLF-NE:} Total Cost of Major Player 1 under Different $\mathbb{E}[u_0^c]$ and $c$ Values}\label{fig:majorcost_comp}
	\end{figure}

 As there are some similarities in the equilibrium results, there are also differences. {In Figure \ref{fig:Major1_comp_3D}, we can see the control chosen by the Major Player 1, $u_1^{**}$ in the Nash Equilibrium (left) and in the Multi-Leader-Follower Nash Equilibrium (middle) given different unit cost of advertisement, $c$ and initial market share, $\mathbb{E}[u_0^c]$ values. In order to compare them, we also plot the difference of the controls chosen by the Major Player 1 in two different equilibria (right). In this 3D plot, we can see that this difference is positive.} Therefore, we note that companies are advertising less in the Multi-Leader-Follower Nash Equilibrium at every level of unit cost of advertisement, $c$ {and initial market share, $\mathbb{E}[u_0^c]$}. However, this difference is more significant at the lower levels of $c$. In other words, in Nash Equilibrium solution there is excessive advertisement (Figure {\ref{fig:Major1_comp_3D}} and \ref{fig:Major1_comp}). Consequently, when the total cost is analysed, it is seen that companies are having a higher cost in the Nash Equilibrium. However this difference is getting  insignificant as unit cost of advertisement increases (Figure \ref{fig:majorcost_comp}). Finally, even if in both cases market shares become closer ($\bar{\mu}$ is closer to 0.5 than $\mathbb{E}[u_0^c]$), it is seen that in Nash Equilibrium solution if $\mathbb{E}[u_0^c]<0.5 (>0.5)$ then also $\bar{\mu}<0.5 (>0.5)$. However, in the Multi-Leader-Follower Nash Equilibrium solution, it can be seen that at the lower levels of the unit cost of advertisement, we end up with $\bar{\mu}>0.5 (<0.5)$ when $\mathbb{E}[u_0^c]<0.5 (>0.5)$. This means that if the cost of advertisement is small enough, in MLF-NE, companies in an adverse position initially have a chance to be the leader of the market with advertising. (Figure \ref{fig:MinControl_comp})

    \begin{figure}[H]
    \centering
			\includegraphics[width=0.65\linewidth, height=0.45\linewidth]{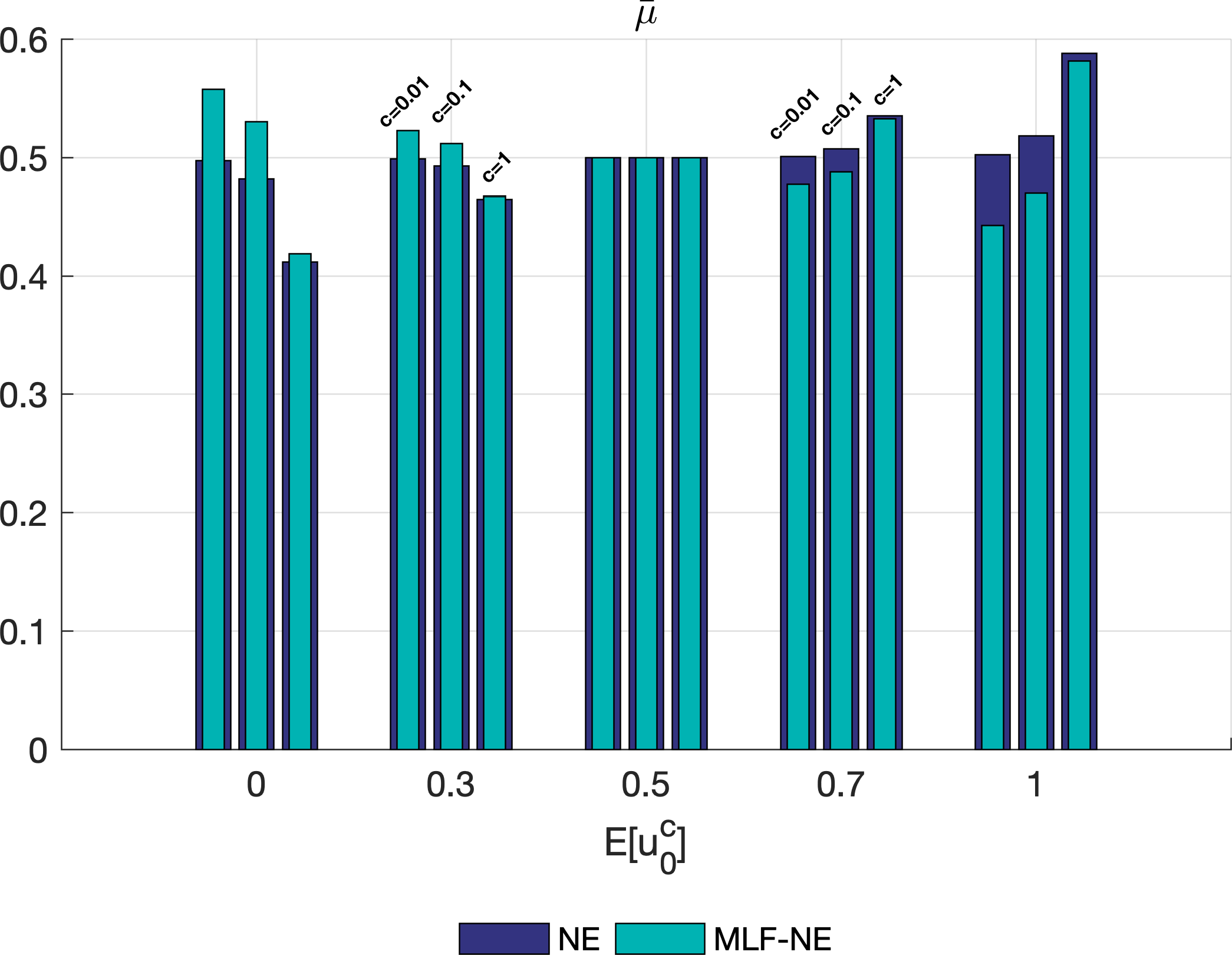}
			\caption{\centering \textbf{NE vs. MLF-NE:} $\bar{\mu}$ under Different $\mathbb{E}[u_0^c]$ and $c=(0.01, 0.1, 1)$ Values}
			\label{fig:MinControl_comp}
	\end{figure}

In conclusion, since companies are advertising significantly higher in Nash Equilibrium at the lower levels of $c$, they are ending up with a significantly higher total cost. Therefore, it can be concluded that it would be the best for companies if they can understand the consumers' behavior and play a 2-player game with the opponent company.

\section{Conclusion}

 In this paper, we analysed two different equilibrium notions for the advertising competition game in a duopoly market when consumers are involved. In the first equilibrium notion (NE), Nash equilibrium is found among all players including major and minor players; on the other hand, in the second equilibrium notion (MLF-NE), major player is assuming minor players are rational and constructing a Nash equilibrium among themselves and they are playing a 2-player game with the opposing company. In this paper, since we have a large number of consumers, mean field game approximation is used in order to find the Nash equilibrium of the minor players. After models are constructed, solution methodologies are given and it is concluded that there exists a unique solution in both cases. Finally, the solutions are compared and it is recommended to the companies in the duopoly to understand the behavior of consumers and use the MLF-NE solution. In this way, they have smaller total costs in the equilibrium by evading from unnecessary levels of advertisement. 

 Main contribution of this report is using the multiple major players and minor players mean field games methodology on the advertisement competition of a duopoly. Even if the model is designed as a static (one-shot) game, the extension to the dynamic version is planned as a future work. \newline

\textit{\textbf{Acknowledgments.}} We would like to thank Dr. Mathieu Lauriere for his valuable comments on the manuscript. Furthermore, we thank the anonymous referee for their valuable feedback that helped us to clarify and improve this manuscript.

\bibliographystyle{siam}

\end{document}